\documentclass[12pt]{amsart}
\usepackage{geometry}
\usepackage{amssymb}
\usepackage{amsmath}
\usepackage{amsfonts}
\usepackage{amsthm}
\usepackage{mathrsfs}
\usepackage{graphicx}
\usepackage[all]{xy}
\usepackage[utf8]{inputenc}
\usepackage[active]{srcltx}
\usepackage{url}
\usepackage{pdfsync}
\usepackage[active]{srcltx}
\usepackage{tabularx}
\usepackage{color}

\DeclareMathOperator{\SL}{SL}
\DeclareMathOperator{\Pic}{Pic}
\DeclareMathOperator{\Hecke}{\mathscr{T}}
\DeclareMathOperator{\M}{M}
\DeclareMathOperator{\GL}{GL}
\DeclareMathOperator{\Tr}{Tr}
\DeclareMathOperator{\Cond}{Cond}
\DeclareMathOperator{\Disc}{Disc}
\DeclareMathOperator{\Gal}{Gal}
\DeclareMathOperator{\Jac}{Jac}

\DeclareMathOperator{\Nm}{Nm}
\DeclareMathOperator{\End}{End}
\DeclareMathOperator{\sign}{sign}
\DeclareMathOperator{\adj}{Adj}
\DeclareMathOperator{\Id}{Id}
\def\id#1{{\mathfrak{#1}}}      
\def \QQ{\mathbb{Q}}

\def \ZZ{\mathbb{Z}}
\def \FF{\mathbb{F}}

\def\Om{\mathscr{O}}
\def\A{\mathscr{A}}
\def\C{\mathscr{C}}
\def \CC{\mathbb{C}}
\def\<#1>{{\left\langle{#1}\right\rangle}}
\def \val{v}

\theoremstyle{plain}
\newtheorem{thm}{Theorem}[section]
\newtheorem{prop}[thm]{Proposition}
\newtheorem{coro}[thm]{Corollary}
\newtheorem{lemma}[thm]{Lemma}

\theoremstyle{remark}
\newtheorem{rem}[thm]{Remark}

\theoremstyle{definition}
\newtheorem{obs/res}[thm]{Observation}

\newtheorem{definition}[thm]{Definition}

\newtheorem{conj}[thm]{Conjecture}
\begin{document}
\title{Heegner points on Cartan non-split Curves}

\author{Daniel Kohen}
\address{IMAS-CONICET, Buenos Aires, Argentina}
\email{kohendaniel@gmail.com}
\thanks{DK was partially supported by a CONICET doctoral fellowship}

\author{Ariel Pacetti}
\address{Departamento de Matem\'atica, Facultad de Ciencias Exactas y Naturales, Universidad de Buenos Aires and IMAS, CONICET, Argentina}
\email{apacetti@dm.uba.ar}
\thanks{AP was partially supported by CONICET PIP 2010-2012 GI and FonCyT BID-PICT 2010-0681.}
\keywords{Cartan Curves, Heegner points}
\subjclass[2010]{Primary: 11G05, Secondary: 11F30}
\begin{abstract} Let $E/\QQ$ be an elliptic curve of conductor $N$, and
  let $K$ be an imaginary quadratic field such that the root number of
  $E/K$ is $-1$. Let $\Om$ be an order in $K$ and assume that there
  exists an odd prime $p$, such that $p^2 \mid\mid N$, and $p$ is inert in
  $\Om$. Although there are no Heegner points on $X_0(N)$
  attached to $\Om$, in this article we construct such points on
  Cartan non-split curves. In order to do that we
  give a method to compute Fourier expansions for forms on Cartan
  non-split curves, and prove that the constructed points form a
  Heegner system as in the classical case.
\end{abstract}
			
\maketitle

\section*{Introduction}
Let $E$ be an elliptic curve over $\QQ$ of conductor $N$. A difficult
unsolved problem is to construct a set of generators for the rational
points on $E$. Nowadays, ``Heegner points'' construction is the only
general method known.  Let $K$ be an imaginary quadratic field
such that $E/K$ has root number $-1$.  Let $\Om$ be an order in $K$ of
discriminant prime to $N$ satisfying the \emph{Heegner hypothesis for
  $X_{0}(N)$}, that is,  all primes dividing $N$ are
  split in $\Om$ (see \cite{Dar04} Hypothesis 3.9). Then, one can
construct points on the modular curve $X_0(N)$ and map them through
the modular parametrization to get points on $E$.  Gross-Zagier
Theorem says that the constructed points are non-torsion if and only
if $L'(E,1) \neq 0$.

Heegner's construction can be generalized to any square-free $N$ using
Shimura curves provided the \emph{Heegner hypothesis for Shimura
  curves} is satisfied: the number of prime numbers dividing $N$ which
are inert in $\Om$ is even. Although the hypothesis might look
awkward, when $N$ is square-free it is the right one for the root
number of $E/K$ to be $-1$. When $N$ is not square-free, this is not
true anymore. For example, suppose $E$ is an elliptic curve over $\QQ$
of conductor $p^2$ ($p$ an odd prime) and $K$ is an imaginary
quadratic field with discriminant $D$ such that $D$ and $p$ are
relatively prime and $p$ is inert in $K$. In this case, the root
number is still $-1$ (see for example \cite[Definition 1.1.3]{ZH}),
but the Heegner hypothesis is not satisfied. Nevertheless, there
should exist some Heegner point construction (and Heegner systems) and
a Gross-Zagier-Zhang formula should hold.  Since there are no Heegner
points on the classical modular curve $X_0(p^2)$ associated to $\Om$
we need to consider other modular curves. A canonical choice in this
case is to consider the so-called Cartan non-split curve, which is a
quotient of the Poincar\'e upper half-plane by a Cartan non-split
group. Since such group is a subgroup of a matrix algebra, once we
proved that our curve $E$ is a quotient of the Cartan non-split
curve Jacobian, the modular parametrization can be explicitly computed using
the Fourier expansion of modular forms for it.

Some new problems appear while working with such groups, for example
what is the right normalization of a modular form? (there is not an
easy formula to relate all Hecke operators eigenvalues with Fourier
coefficients of eigenforms for such groups). Some interesting problems
that will not be addressed in this article (and are unknown in
general) are determining the strong Weil curve for the Cartan
non-split curve (even deciding when it coincides with the strong Weil
curve for $\Gamma_0(N)$), and determining the Manin constant for it.

In this article we show how to compute Hecke operators for Cartan
non-split curves (and curves which are mixed situations of classical
curves for some primes and Cartan non-split for the other ones) and how to
compute the Fourier expansion of Cartan modular forms. We propose a
natural normalization (well defined up to $\pm 1$) and show how to
construct Heegner points (and Heegner systems) on modular curves over
imaginary quadratic fields satisfying the Cartan-Heegner hypothesis
using the presented theory.

The present article is organized as follows: we start with the case
$N=p^2$ where all new ideas appear while avoiding to deal with subindices in a
first reading. In the first section we recall the basic definitions of
Cartan non-split curves, and give a moduli space interpretation for
them. Our moduli problem is different from the classical one and also
from the one presented in \cite{Rebolledo}, but it makes the geometric
and analytic properties of Hecke operators and Heegner systems more
clear. For example, with this moduli interpretation it is easy to
define Hecke operators (outside $p$), and show that this definition
agrees with the double coset definition (as in \cite{Chen}). It is
also easily generalizable to the mixed situations.

Next we focus on the problem of computing Fourier expansions of Cartan
modular forms. We propose a suitable normalization and prove that with
this normalization, the Fourier expansion of a Cartan modular form has
coefficients in $\QQ(\xi_p)$ (the $p$-th cyclotomic field). The way to
compute the Fourier expansion is to write the form as a linear
combination of other modular forms (twists of the weight $2$ modular form
attached to $E$), and then
solve a linear system to compute the combination explicitly. A
Theorem of Chen and Edixhoven (\cite{Chen,Edi}) proves that our curve
is isogenous to a quotient of the Jacobian of the Cartan non-split
curve, so the Eichler-Shimura construction and the Abel-Jacobi map
give the modular parametrization. A difference with the $\Gamma_0(N)$
case is that the cusps for the Cartan non-split curve are not defined
over $\QQ$ (implying the natural modular parametrization is not rational),
so we average over all conjugate cusps to get a rational map (which we
also call modular parametrization). Galois conjugation
sends a Cartan modular form to another Cartan modular form but
for another Cartan subgroup, i.e. it corresponds to another choice of
a non-square modulo $p$, so in the modular parametrization all Cartan
non-split groups are involved.

After the theory for level $p^2$  is done, we move to the general
case of mixed types, i.e. elliptic curves whose conductor are not
square-free, some primes dividing the conductor are split in $\Om$
while other ones are inert. Although no extra difficulties appear,
we believe that considering first the conductor $p^2$ case gives a
better understanding of the new ideas involved.

The third and fourth sections are about constructing Heegner points
and Heegner systems satisfying the usual compatibility
relations. Using this, we can prove a big part of the Birch and
Swinnerton-Dyer conjecture for $E/K$ by applying the usual
Darmon-Kolyvagin and Gross-Zagier-Zhang formula machinery
 (Theorem~\ref{thm:Kolyvagin} and Theorem~\ref{thm:GZZ} ). We also
include some computational details on how our construction can be
carried out for any particular curve $E$ and any order $\Om$ satisfying the
Cartan-Heegner hypothesis.

The last section of this article contains many examples where we show
how the method works for different elliptic curves, including the
Manin constants and Heegner points obtained by our method for each of
them.

\medskip

\noindent{\bf Acknowledgments:} We would like to thank Professor Henri
Darmon for many suggestions and for the discussions the first author
had with him while visiting McGill University. We also
thank Professor Imin Chen and Professor Bas Edixhoven for explaining some technicalities of the
definition of Hecke operators on Cartan non-split groups and Professor
Tim Dokchitser for the results in Appendix~\ref{app:1}. We also thank
the referee and Juan Restrepo for useful comments and corrections.

\section{Cartan non-split curves of prime level}
\noindent{\bf Notations and conventions:} throughout this article, $p$
will denote an odd prime and $\varepsilon$ will be a non-square modulo
$p$.  Given a matrix $A \in \M_{2\times 2}(\ZZ)$,
$\bar{A}$ will denote its reduction modulo $p$.
\subsection{Definition} 
The Cartan non-split ring modulo $p$ is the ring
 \[ C^{\varepsilon}_{ns}(p) = \left\{ \left( \begin{array}{ccc}
a & b \\
c & d  \\ \end{array} \right) \ \in \\M_{2 \times 2}(\mathbb{F}_{p}) : a \equiv d , c \equiv  b\varepsilon \bmod{p}   \right\}.
\]
The group of invertible elements $({C^{\varepsilon }_{ns}}(p))^{\times}$ is isomorphic to the cyclic group
$\mathbb{F}^{\times}_{p^2}={\mathbb{F}_{p}(\sqrt{\varepsilon})}^{\times}$.
 We also define the ring
$M^{\varepsilon}_{ns}(p)=\left\{A \in \M_{2\times 2}(\ZZ) \; : \; \bar{A}  \in  C^{\varepsilon}_{ns}(p) \right\}$.
The  Cartan non-split group  $\Gamma_{ns}^\varepsilon(p)$ is the group of determinant
  $1$ matrices in $M^{\varepsilon}_{ns}(p)$.
 We can also consider
\[ C^{\varepsilon +}_{ns}(p) = \left\{ \left( \begin{array}{ccc}
a & b \\
c & d  \\ \end{array} \right) \in M_{2 \times 2}(\mathbb{F}_{p}) : a \equiv d , c \equiv  b\varepsilon  \ \text{or}  \ a \equiv -d , c \equiv  -b\varepsilon \right\}.\]
and define $M^{\varepsilon +}_{ns}(p)$ and $\Gamma_{ns}^{\varepsilon +}(p)$ as before. The group  $\Gamma_{ns}^{\varepsilon +}(p)$ is called the normalizer of the Cartan 
non-split group.

Let $\mathcal{H}$ be the Poincar\'e  half plane, and consider the complex curve
$Y^{\varepsilon}_{ns}(p) = \Gamma^{\varepsilon}_{ns}(p) \backslash
\mathcal{H}$  whose compactification obtained by adding the cusps is the Cartan non-split modular curve of level $p$, $X^{\varepsilon}_{ns}(p)= \Gamma^{\varepsilon}_{ns}(p) \backslash
{\mathcal{H}^{*}}$.
Analogously, we can define
$X^{\varepsilon +}_{ns}(p)= \Gamma^{\varepsilon +}_{ns}(p) \backslash
{\mathcal{H}^{*}}$.
Since
$\det: ({C^{\varepsilon }_{ns}}(p))^{\times} \rightarrow \mathbb{F}^{\times}_{p}$ is
surjective, the modular curves $X^{\varepsilon}_{ns}(p)$ and
$X^{\varepsilon +}_{ns}(p)$ are both defined over $\QQ$ (see \cite{Shimura} Section $6.4$ and Proposition $6.27$).
\subsection{Moduli interpretation}
We will give a new moduli interpretation for the complex points of the Cartan non-split curve. For other moduli interpretations see \cite{Serre} Appendix $5$ and \cite{Rebolledo}.
Consider pairs $(E,\phi)$, where $E/\CC$ is an elliptic curve and
$\phi \in \End_{\FF_p}(E[p])$ satisfies that ${\phi}^2$ is
multiplication by $\varepsilon$. We identify two such pairs $(E,\phi)$,
$(E', \phi')$ if there exists an isomorphism of
elliptic curves $\Psi: E \to E'$ such that the following diagram is
commutative:
\[
\xymatrix{ E[p]
 \ar[rr]^{\phi}\ar[d]_{ \Psi }& &  E[p]  \ar[d]^{ \Psi }\\
E'[p] \ar[rr]_{\phi'}& & E'[p]}
\]

For any number field $K$, we say that the point $(E,\phi)$ is a $K$-rational point of the Cartan non-split curve if $E$ is an elliptic curve defined over $K$ and $\phi$ is defined 
over $K$. Recall that by definition $\phi$ is defined over $K$ if $\phi^{\sigma}=\phi$ for every $\sigma \in \Gal(\bar{K}/K)$, i.e. $\phi(P^{\sigma})=\phi(P)^\sigma$ for every
$P \in E[p]$ and every $\sigma \in \Gal(\bar{K}/K)$.
\begin{prop}
\label{prop:moduliinterpretation}
The moduli problem of pairs $(E,\phi)$ is represented by the Cartan
non-split curve $Y_{ns}^{\varepsilon}(p)$. The point
$\Gamma^{\varepsilon}_{ns}(p) \tau$ corresponds to the pair
$(E_{\tau},\phi_\tau)$, where
$E_{\tau}=\CC / \left\langle \tau,1 \right\rangle $ and $\phi_\tau$ is
the endomorphism of $E_{\tau}[p]$ whose matrix in the basis
$B_{\tau}=\left\{ \frac{1}{p} , \frac{\tau}{p}\right\} $ equals
$\left( \begin{smallmatrix}
    0 & 1\\
    \varepsilon & 0 \end{smallmatrix} \right)$.
\end{prop}

Before writing the proof, we need an auxiliary Lemma.

\begin{lemma}
  Let $M \in \GL_2(\FF_p)$ satisfying $M^2 = \left(\begin{smallmatrix}
      \varepsilon & 0 \\ 0 & \varepsilon \end{smallmatrix} \right)$. Then, 
  there exists $A \in \SL_2(\ZZ)$ such that $\bar{A} M \bar{A}^{-1} =
  \left(\begin{smallmatrix} 0 & 1\\ \varepsilon & 0\end{smallmatrix}
  \right)$.
\label{lemma:conjugationrepresentative}
\end{lemma}

\begin{proof}
  Clearly there exists $B \in \GL_2(\FF_p)$ such that $B^{-1}MB =
  \left(\begin{smallmatrix}0 & 1\\ \varepsilon & 0\end{smallmatrix}
  \right)$. Consider the centralizer of
  $\left(\begin{smallmatrix}0 & 1\\ \varepsilon & 0\end{smallmatrix}
  \right)$, which is given by $C_{ns}^{\varepsilon  \times}(p)$ and take any
  matrix $C$ of determinant $\det(B)^{-1}$ there. Then $BC \in \SL_2(\FF_p)$ and $(BC)^{-1}M (BC)= \left(\begin{smallmatrix}0 & 1\\ \varepsilon & 0\end{smallmatrix}
  \right)$. The result follows from the fact that the reduction map
  $\SL_2(\ZZ) \mapsto \SL_2(\FF_p)$ is surjective.
\end{proof}

\begin{proof}[Proof of Proposition 1.2]
  We need to check 
that the previous correspondence between points on
    $Y_{ns}^\varepsilon(p)$ and pairs $(E,\phi)$ is well defined
	and bijective.

        Let $\tau$ and $\tau'$ be points on $\mathcal{H}$
        corresponding to pairs
        $(E_{\tau},\phi_\tau),(E_{\tau '},\phi_{\tau'})$
        respectively. To prove that the map is well defined and
        injective is enough to prove that such pairs are isomorphic
        if and only if $\tau$ and $\tau'$ are equivalent under
        $\Gamma^{\varepsilon}_{ns}(p)$.  It is well known that any morphism $\Psi$ between
        two elliptic curves is given by multiplication by a complex
        number $\alpha$.  In particular, if $\Psi$ is an isomorphism,
        $\alpha\<\tau,1> = \<\tau',1>$, so that there exists
        $\left( \begin{smallmatrix}
            a & b \\
            c & d \end{smallmatrix} \right) \in \SL_{2}(\ZZ)$
        such that $\alpha \tau= a\tau '+b$ and $\alpha
        =c\tau'+d$. Moreover,
	$\Psi$ must satisfy $\phi_{\tau'}
  \Psi= \Psi \phi_\tau$. In the chosen basis, this is equivalent to:
  \begin{itemize}
  \item $\Psi(\phi_\tau(\frac{1}{p})) = \phi_{\tau'}(\Psi(\frac{1}{p}))$.
  \item $\Psi(\phi_\tau(\frac{\tau}{p})) = \phi_{\tau'}(\Psi(\frac{\tau}{p}))$.
  \end{itemize}

It is easy to see that the following holds:
\begin{itemize}
\item $\Psi(\phi_{\tau}(\frac{1}{p}))=\Psi(\frac{\tau \varepsilon }{p} )= \frac{
    \alpha \tau \varepsilon }{p}= \frac{ a  \varepsilon \tau' +b \varepsilon }{p}$.
\item $\Psi(\frac{1}{p})=\frac{\alpha}{p}=\frac{c\tau'+d}{p}$, so $\phi_{\tau'}(
  \Psi(\frac{1}{p}))= \frac{ c+ d  \varepsilon \tau '}{p}.$ 
\end{itemize}

Since equality holds modulo $\left\langle 1, \tau' \right\rangle$,
we get that $a \equiv d \bmod{p}$ and $c \equiv \varepsilon b
\bmod{p}$. This proves that the pairs $(E_\tau,\phi_\tau)$ and
$(E_{\tau'},\phi_{\tau'})$ are isomorphic by a map satisfying the
first condition if and only if $\tau$ and $\tau'$ are equivalent under
$\Gamma^{\varepsilon}_{ns}(p)$.\\
The commutative condition for the second basis elements is
similar and gives the same constraint.

To prove surjectivity, let $(E,\phi)$ be any pair as before. Up to
isomorphism we can assume that $E= \CC/\<\tau,1>$, where
$\tau \in \mathcal{H}$. Let
$B=\left\{ \frac{1}{p},\frac{\tau}{p}\right\}$ be a basis of
$E[p]$. By Lemma ~\ref{lemma:conjugationrepresentative}, there exists
a matrix $A\in \SL_2(\ZZ)$ such that
$\bar{A}[\phi]_B \bar{A}^{-1} = \left(\begin{smallmatrix}0 & 1\\
    \varepsilon & 0\end{smallmatrix} \right)$.
Hence, $[E,\phi] \simeq [\CC/\<A\cdot \tau,1>,\phi']$, where
$[\phi']_{\left\{\frac{1}{p},\frac{A\cdot \tau}{p}\right\}} =
\left(\begin{smallmatrix}0 & 1\\ \varepsilon & 0\end{smallmatrix}
\right)$. 
\end{proof}

\begin{rem}
  The moduli problem for $\Gamma^{\varepsilon +}_{ns}(p)$ consists on
  pairs $(E,\phi)$ as before, where two pairs $(E,\phi), (E',\phi')$ are isomorphic if
  $\Psi \phi = \pm \phi' \Psi$. There is an involution $\omega^{\varepsilon}_{p}$ acting on
  $X^{\varepsilon}_{ns}(p)$ given by
  $\omega^{\varepsilon}(E,\phi)=(E, -\phi)$, and 
  $X^{\varepsilon
    +}_{ns}(p)=X^{\varepsilon}_{ns}(p)/\omega^{\varepsilon}_{p}$.
	\label{rem:involution}
\end{rem}

\subsection{Modular forms and Hecke operators}
Let $\Gamma \subset \SL_2(\ZZ)$  be a congruence subgroup. Let
$f:\mathcal{H} \to \CC$ be an holomorphic function.  If
$\left(\begin{smallmatrix}a & b\\ c & d\end{smallmatrix} \right) \in
\Gamma$, and $k \in \ZZ$, we define the slash operator
\[
f|_k[\left(\begin{smallmatrix}a & b\\ c & d\end{smallmatrix} \right)](z) = (cz+d)^{-k} f\left(\frac{az+b}{cz+d}\right).
\]
Let $M_k(\Gamma)$ be the space of holomorphic functions which are
invariant under the previous action for all elements in $\Gamma$ and
which are holomorphic at all the cusps, and let $S_k(\Gamma)$ be the
subspace of cusp forms, i.e. those forms in $M_k(\Gamma)$ whose
$q$-expansions at all the cusps have zero constant coefficient. Let
$\Gamma(p)$ be the principal congruence subgroup of level $p$. The inclusion
$\Gamma(p) \subset \Gamma^{\varepsilon}_{ns}(p)$  gives a
reverse inclusion at the level of modular forms
$S_{2}(\Gamma^{\varepsilon}_{ns}(p)) \subset S_{2}(\Gamma(p))$.  If
$\alpha_{p}= \left( \begin{smallmatrix}
    p & 0 \\
    0 & 1 \\ \end{smallmatrix} \right)$
and $f \in S_{2}(\Gamma(p))$, $\tilde{f}=f |_2 [\alpha_{p}]$ is a
modular form with respect to
$({\alpha_{p}})^{-1} \Gamma(p) {\alpha_{p}}=\Gamma_{0}(p^{2})\cap
\Gamma_{1}(p)$.
Define $\tilde{\Gamma}(p) := \Gamma_{0}(p^{2})\cap
\Gamma_{1}(p)$.
Thus, slashing by $\alpha_{p}$ gives the isomorphism
$S_{2}(\Gamma(p)) \cong S_{2}(\tilde{\Gamma}(p))$.

There are two ways to define Hecke operators for classical
subgroups. The geometric way is to define them as correspondences on  the
modular curve and, via the moduli interpretation, translate this
action to an action on modular forms. The algebraic way is to define them
in terms of double coset operators. We will describe both definitions and we will prove that they agree.

\subsubsection{Geometric definition}
Let $n$ be a positive integer prime to $p$ and $(E,\phi)$ be a pair 
corresponding to a point on the moduli interpretation of the curve
$Y_{ns}^\varepsilon(p)$. Define the Hecke
operator \label{Heckegeometric}
\[
\Hecke^{\varepsilon}_n((E,\phi)) := \sum_{\psi:E \to E'}\left(E',\frac{1}{n}\psi \circ \phi \circ \hat{\psi}\right),
\]
where the sum is over degree $n$ isogenies $\psi:E \to E'$ of cyclic kernel, and
$\hat{\psi}$ denotes the dual isogeny. Note that since $\gcd(n,p)=1$,
$\frac{1}{n} \in \End_{\FF_p}(E'[p])$. Also, since $\psi \circ
\hat{\psi}$ and $\hat{\psi} \circ \psi$ are multiplication by $n$,
$(\frac{1}{n}\psi \circ \phi \circ \hat{\psi})\circ(\frac{1}{n}\psi
\circ \phi \circ \hat{\psi})$ is multiplication by $\varepsilon$,
so the points in the formula belong to $Y_{ns}^\varepsilon(p)$.

\subsubsection{Algebraic definition} We make a little survey of Hecke
operators for the Cartan non-split curve, following Shimura's book
\cite{Shimura}. Define
\[
\Delta_p:=\left\{A \in \M_{2\times 2}(\ZZ)\, : \, \det(A) >0 \text{ and } \gcd(p,\det(A))=1\right\},
\]
and  $\Delta^{\varepsilon}_{ns}(p):= \Delta_p \cap M^{\varepsilon}_{ns}(p)$. Moreover, consider
\[
\Delta(p):=\left\{A \in \Delta_p \, : \, \bar{A} \equiv \left(\begin{smallmatrix} 1 & 0 \\ 0 & * \end{smallmatrix} \right) \bmod p\right\}.
\]
Let $R(\Gamma^{\varepsilon}_{ns}(p),\Delta^{\varepsilon}_{ns}(p))$ and $R(\Gamma(p),\Delta(p))$ be the 
Hecke rings as defined in \cite{Shimura} page $54$. 

We need to introduce a new operator. Let $n \in \ZZ$ satisfy $p \nmid n$ and
let $B \in \Delta_{ns}^\varepsilon(p)$ be any matrix with determinant
congruent to $n$ modulo $p$. Let $A^{\varepsilon}_n \in \SL_2(\ZZ)$ be
such that
$A^{\varepsilon}_n \equiv B \left(\begin{smallmatrix} 1 & 0 \\ 0 &
    1/n \end{smallmatrix} \right) \bmod p$.
The action of $A^{\varepsilon}_n$ on
$S_{2}( \Gamma_{ns}^{\varepsilon}(p) )$ defines an operator that we
will denote $\upsilon^{\varepsilon}_n$.

\begin{lemma}
  The operator $\upsilon^{\varepsilon}_n$ defines an isomorphism from  $S_2(\Gamma_{ns}^\varepsilon(p))$ to
  $S_2(\Gamma_{ns}^{\varepsilon n^2}(p))$ which depends only on the class of $n$ modulo $p$. It is equal to the double coset
	operator $\Gamma_{ns}^\varepsilon(p) A^{\varepsilon}_{n} \Gamma_{ns}^{\varepsilon n^2}(p)$.
  \label{lemma:matA}
\end{lemma}

\begin{proof}
  Since $\bar{B} \in C_{ns}^\varepsilon(p)$, and
  $\left( \begin{smallmatrix} 1& 0 \\ 0 & 1/n \end{smallmatrix}
  \right)^{-1} C_{ns}^\varepsilon(p) \left( \begin{smallmatrix} 1& 0
      \\ 0 & 1/n \end{smallmatrix} \right) = C_{ns}^{\varepsilon
    n^2}(p)$,
  the first assertion follows. Let $B$ and $B'$ be matrices in
  $\Delta_{ns}^\varepsilon(p)$ of determinant $n$ and $n'$
  respectively with $n \equiv n' \bmod{p}$. Choose any two
  matrices $A^{\varepsilon}_{n}$ and $A^{\varepsilon}_{n'}$
  corresponding to $B$ and $B'$ respectively. Clearly
  $A^{\varepsilon}_{n} {A^{\varepsilon}_{n'}}^{-1} \in
  \Gamma_{ns}^\varepsilon(p)$, therefore, this matrix  acts trivially.
\end{proof}

Let
$h:R(\Gamma(p),\Delta(p)) \rightarrow
R(\Gamma^{\varepsilon}_{ns}(p),\Delta^{\varepsilon}_{ns}(p))$
be the map given by
$\Gamma(p) \beta \Gamma(p) \mapsto \Gamma^{\varepsilon}_{ns}(p)
A_{\det(\beta)}^{\varepsilon} \beta \Gamma^{\varepsilon}_{ns}(p)$.

\begin{prop}
  The map  $h$ is an isomorphism of Hecke rings.
\label{lemma:heckeiso}
\end{prop}

\begin{proof} We have a map 
  $h_{1}:R(\Gamma^{\varepsilon}_{ns}(p),\Delta^{\varepsilon}_{ns}(p)) \rightarrow R(\SL_2(\ZZ),\Delta_p)$ given by
  \[\Gamma^{\varepsilon}_{ns}(p) \alpha \Gamma^{\varepsilon}_{ns}(p) \mapsto \SL_2(\ZZ) \alpha
  \SL_2(\ZZ),\]
  and a map  $h_{2}:R(\Gamma(p),\Delta(p)) \rightarrow R(\SL_2(\ZZ),\Delta_p)$ given by
  \[\Gamma(p) \beta \Gamma(p) \mapsto \SL_2(\ZZ) \beta 
  \SL_2(\ZZ).\]
  Both maps are easily seen to be isomorphisms of Hecke rings by the same proof used in \cite{Shimura} Proposition $3.31$. Moreover, the  map $h=h_{1}^{-1}h_{2}$ is given by
  $\Gamma(p) \beta \Gamma(p)  \mapsto  \Gamma^{\varepsilon}_{ns}(p)  A_{\det(\beta)}^{\varepsilon} \beta  \Gamma^{\varepsilon}_{ns}(p) $ and gives the desired isomorphism.
\end{proof}

We can consider the classical Hecke operators $T_{n}$ acting on $S_{2}(\tilde{\Gamma}(p))$ for $n$ relatively prime to $p$. Slashing by $\alpha_{p}$ we obtain
the corresponding Hecke operator $T_{n}$ acting on $S_{2}(\Gamma(p))$. In view of the above proposition we define  the Hecke operator  $\Hecke^{\varepsilon}_{n}  \in R(\Gamma^{\varepsilon}_{ns}(p),\Delta^{\varepsilon}_{ns}(p))$  as the operator  $h(T_{n})$.

\begin{lemma}
If $\beta \in \Delta(p)$, the operator $\Gamma(p) \beta \Gamma(p)$ acting on $S_{2}(\Gamma_{ns}^{\varepsilon n^2}(p))$ is equal to the operator $\Gamma^{\varepsilon n^2}_{ns}(p) \beta \Gamma^{\varepsilon }_{ns}(p)$.
\label{lemma:Tn}
\end{lemma}
\begin{proof}
Mimics the proof of Lemma \ref{lemma:matA}.
\end{proof}

\begin{prop} As operators on $S_{2}(\Gamma_{ns}^\varepsilon(p))$, $\Hecke^{\varepsilon}_n =  T_n \circ \upsilon^{\varepsilon}_{n}.$
\label{prop:heckedecomposition}
\end{prop}
\begin{proof}
 Proposition $3.7$ of \cite{Shimura} says that 
 \[ \Gamma^{\varepsilon}_{ns}(p)   A_{\det(\beta)}^{\varepsilon} \beta \Gamma^{\varepsilon}_{ns}(p)=\Gamma_{ns}^\varepsilon(p) A^{\varepsilon}_{n} \Gamma_{ns}^{\varepsilon n^2}(p)
\Gamma^{\varepsilon n^2}_{ns}(p) \beta \Gamma^{\varepsilon}_{ns}(p). \]
 Therefore, the result follows from Lemma \ref{lemma:matA} and Lemma \ref{lemma:Tn}.
\end{proof}

\begin{coro}
 If $n \equiv 1 \bmod{p}$,  $\Hecke_{n}^\varepsilon=T_{n}.$
\label{coro:sameHecke}
\end{coro}
\begin{proof}
Since the matrix $A^{\varepsilon}_{n}$ can be taken to be the identity, $\upsilon^{\varepsilon}_{n}$ is the identity map.
\end{proof}

Exactly in the same way as Proposition \ref{prop:heckedecomposition} we can prove the following.
\begin{prop} For any $n$ prime to $p$, the operators 
$T_n: S_{2}(\Gamma^{\varepsilon}_{ns}(p)) \to
  S_{2}(\Gamma^{\varepsilon/{n}^2}_{ns}(p))$ and $\upsilon^{\varepsilon}_n : S_{2}(\Gamma^{\varepsilon}_{ns}(p)) \to
  S_{2}(\Gamma^{\varepsilon n^2}_{ns}(p))$
  are morphisms of Hecke modules.
\label{prop:heckeaction}
\end{prop}
\begin{thm}
  The geometric and algebraic definitions of Hecke operators coincide.
\label{prop:heckeoperatorscompatibility}
\end{thm}
\begin{proof} 
  We can restrict to $n$ prime and $n \neq p$. It is enough to see
  that the set of representatives used in one definition can be taken
  as representatives for the other one.  Take representatives for
  $\Gamma^{\varepsilon}_{ns}(p) A^{\varepsilon}_{n}
  \left(\begin{smallmatrix} 1 & 0 \\ 0 & n \end{smallmatrix} \right)
  \Gamma^{\varepsilon}_{ns}(p)$
  modulo $\Gamma^{\varepsilon}_{ns}(p)$. By \cite{Shimura} Lemma
  $3.29$ part $(5)$ these are also representatives for
  $\SL_2(\ZZ) A^{\varepsilon}_{n} \left(\begin{smallmatrix} 1 & 0 \\ 0
      & n \end{smallmatrix} \right) \SL_2(\ZZ)=\SL_2(\ZZ)
  \left(\begin{smallmatrix} 1 & 0 \\ 0 & n \end{smallmatrix} \right)
  \SL_2(\ZZ)$
  modulo $\SL_2(\ZZ)$. This set of representatives coincides with a
  set of representatives of cyclic isogenies of degree $n$. Each
  representative is a matrix $A$ of determinant $n$. The dual isogeny
  is given by the matrix $\adj(A)$. Both matrices belong to
  $\Delta^{\varepsilon}_{ns}(p)$, thus, they commute with the matrix
  $\left(\begin{smallmatrix} 0 & 1 \\ \varepsilon &
      0 \end{smallmatrix} \right)$
  modulo $p$ and $A \adj(A)=n \Id$. Therefore, recalling the
  geometric definition we have that
  $ \Hecke^{\varepsilon}_n([\CC/\<\tau,1>,\phi_\tau])=
  \sum_{A}[\CC/{\scriptstyle\<{A \tau},1>},\phi_{A}]$,
  where
  $[\phi_A]_{\left\{\frac{1}{p}, \frac{A \tau}{p}\right\}} =
  \left(\begin{smallmatrix} 0 & 1 \\\varepsilon & 0 \end{smallmatrix}
  \right)$ as desired.
   \end{proof}
\subsection{ Chen-Edixhoven isogeny theorem}
If $\C$ is a curve, we denote by $\Jac(\C)$ its Jacobian.
\begin{thm}[Chen-Edixhoven]
 The new part of  $\Jac(X^{+}_{0}(p^2))$  is isogenous to $\Jac(X^{\varepsilon +}_{ns}(p))$. Furthermore, the new part of $\Jac(X_{0}(p^2))$ and $\Jac(X^{\varepsilon}_{ns}(p))$ are isogenous. In addition, the isogenies are Hecke equivariant.

\label{thm:chen}
\end{thm}
 
\begin{proof}
  See Theorem 1 of \cite{Chen} , Theorem 1.1 of \cite{Edi} and Theorem
  2 of \cite{SmEd}. Although the Hecke equivariant condition is not
  explicitly stated, by Theorem 2 of \cite{SmEd} the decompositions
  are functorial in $(M ,\alpha)$, hence they are preserved by all
  endomorphisms of $M$ that commute with the $G$-action. In the case
  of Jacobians of modular curves this means that the isogenies commute
  with all Hecke operators of level relatively prime to $p$. 
\end{proof}
  In particular, if we start with a normalized newform
  $g \in S_2(\Gamma_0(p^2))$ such that $T_{n} g= \lambda_{n}g$ for all
  $n$ relatively prime to $p$, Theorem~\ref{thm:chen} implies the existence of a
 form
  $g_{\varepsilon} \in S_2(\Gamma_{ns}^\varepsilon(p))$ such that
  $\Hecke^{\varepsilon}_{n} g_{\varepsilon}= \lambda_{n}
  g_{\varepsilon}$
  for all $n$ relatively prime to $p$. 
  
  Chen-Edixhoven's Theorem plus Multiplicity one for classical newforms (Theorem $5.8.2$ of \cite{MR2112196})
  for $S_2(\Gamma_0(p^2))$ give multiplicity one for a system of
  eigenvalues for the Hecke algebra $R(\Gamma^{\varepsilon}_{ns}(p),\Delta^{\varepsilon}_{ns}(p))$.

  Our primary goal is to compute the Fourier expansion of
  $g_{\varepsilon}$. Since $g_\varepsilon$ is an eigenfunction for the
  Hecke operators $\Hecke^{\varepsilon}_{n}$, and since such operators
  coincide with $T_n$ if $n \equiv 1 \pmod p$, our form lies in the
  space of eigenfunctions for $T_n$ with eigenvalues $\lambda_n$ for
  $n \equiv 1 \pmod p$. Then we can write
  $\tilde{g_{\varepsilon}}=g_\varepsilon|_2[\alpha_p]$ as a linear
  combination of eigenforms on $S_2(\widetilde{\Gamma(p)})$ which have
  the same eigenvalues as $g$ for $n \equiv 1 \pmod p$.  

Let $\A = \{ f \in S_2(\widetilde{\Gamma(p)})  \, \text{eigenform} \,: \, \lambda_n(g) = \lambda_n(f) \text{ for all }n \equiv 1 \pmod p\}$. Can we characterize $\A$?

\medskip

In general, if $\chi$ is a character modulo $N$ and
$h \in S_2(\Gamma(N))$ is a newform we denote by $h \otimes \chi$ the
\emph{twist} of $h$ by $\chi$.  If $T_{n}h= \lambda_{n}(h) h $ then
$T_{n} (h \otimes \chi )= \lambda_{n}(h) \chi(n) (h \otimes \chi)$.
This implies that if $\chi$ is a character modulo $p$,
  $g \otimes \chi \in \A$.
\begin{thm}
\label{thm:rajan}
Let $f \in S_{2}(\Gamma_{0}(p^2), \psi)$ be an eigenform for the
classical Hecke algebra, where $\psi$ is a character
modulo $p$. Let $g \in S^{new}_{2}(\Gamma_0(p^2))$ be an
eigenform without complex multiplication, and suppose that $f$ and $g$ have the same
eigenvalues on a set of primes of positive upper density. Then, there
exists a Dirichlet character $\chi$ modulo $p$ such that the
eigenforms $g \otimes \chi$ and $f$ have the same eigenvalues at all
but a finite number of primes.
\end{thm}
\begin{proof}
 See  \cite{rajan1998strong} Corollary $1$.
\end{proof}

Therefore, all elements of $\A$ are the form $g \otimes \chi$
(where $\chi$ varies over the characters of conductor $p$) or are 
newforms attached to them since it may happen that
$g \otimes \chi \in S_{2}(\Gamma_{0}(p^2), \chi^2)$ is not a newform
(it may fail to be new at $p$).  In that case there is an associated
newform living on $S_{2}(\Gamma_{0}(p), \chi^2)$ that appears in the
linear combination as well. Being new at $p$ can be read from the
type of the local automorphic representation of $g$ at the prime $p$,
 as explained in \cite{Atkin-Li}. We have proved the following Theorem:

\begin{thm}
  Let $g \in S^{new}_2(\Gamma_0(p^2))$ be a normalized eigenform  with
  eigenvalues $\lambda_{n}$ ($n$ relatively prime to $p$), and suppose that $g$ does not have complex multiplication. Let $g_{\varepsilon} \in
  S_2(\Gamma^{\varepsilon}_{ns}(p))$ be the unique normalized
  eigenform such that $\Hecke^{\varepsilon}_{n} g_{\varepsilon}=\lambda_{n}g_{\varepsilon}$ ($n$ relatively prime to $p$).
  Let $\pi_{p}$ be the local automorphic representation of $g$ at $p$. 
  \begin{itemize}
  \item If $\pi_{p}$ is supercuspidal, then  $g \otimes \chi$ is a newform for $S_2(\Gamma_0(p^2),\chi^2)$ for all characters $\chi$ modulo $p$ and
\[
\tilde{g_{\varepsilon}} = \sum_{\chi} a_{\chi} (g \otimes \chi),
\]
for some  $a_{\chi} \in \CC$, 
where the sum is over all  characters modulo $p$.
\item If $\pi_{p}$ is Steinberg, there exists a newform $h \in
  S_{2}(\Gamma_{0}(p))$ such that $h \otimes \varkappa_p =g$, where $\varkappa_p$ is
  the quadratic character modulo $p$, and
\[
\tilde{g_{\varepsilon}}= \sum_{\chi} a_{\chi} (g \otimes \chi)+ ah,
\]
for some  $a_{\chi},a \in \CC$, 
where the first sum is over all  characters modulo $p$.
\item If $\pi_{p}$ is a ramified Principal Series, there exists a
  non-quadratic character $\theta_{p}$ modulo $p$ and newforms $h \in
  S_{2}(\Gamma_{0}(p), \overline{\theta_{p}}^2)$, $\bar{h} \in
  S_{2}(\Gamma_{0}(p), \theta_{p}^2)$ such that $h \otimes \theta_{p}=g=\bar{h}
  \otimes \bar{\theta_{p}}$. Then
\[
\tilde{g_{\varepsilon}}= \sum_{\chi} a_{\chi} (g \otimes \chi)+ a_{1}h+ a_{2} \bar{h} ,
\]
for some  $a_{\chi},a_{1},a_{2} \in \CC$, 
where the first sum is over all  characters modulo $p$. 

\end{itemize}

\label{thm:combination}
\end{thm}

\begin{rem}
  If $g_{\varepsilon} \in S_2(\Gamma^{\varepsilon +}_{ns}(p))$ is an eigenform, then
  $T_{n}f =\lambda_{n} f$ for $n \equiv -1 \bmod{p}$ as well. Therefore
  all the non-zero coefficients in the linear combination of
  Theorem \ref{thm:combination} are those corresponding to even
  characters.

Similarly, if $g_{\varepsilon} \in S_2(\Gamma^{\varepsilon -}_{ns}(p))$ (i.e. any matrix in
  the normalizer but not in the Cartan itself acts as $-1$), then the non-zero coefficients in the linear
  combination of Theorem~\ref{thm:combination} are those corresponding
  to odd characters.
 \label{rem: evenandodd}
\end{rem}

\subsection{ Fourier expansions}
In order to compute the Fourier
expansion of the normalized newform  $g_{\varepsilon}$
we first need to understand the action of the Galois group
$\Gal(\CC/\QQ)$ on modular forms.
For $a \in {\QQ}^2$ and $z \in \mathcal{H}$ define
\[ 
f_{a}(z)=\frac{g_{2}(z,1)g_{3}(z,1)}{\Delta(z,1)} \wp(a\left( \begin{smallmatrix}
      z  \\
      1  \\ \end{smallmatrix} \right); z,1 ),
\]
where $\wp( - ; \omega_{1},\omega_{2})$ is the classical Weierstrass
function associated to the lattice $L=\left\langle \omega_{1},
  \omega_{2}\right\rangle$; $g_{2}(L)=60 G_{4}(L)$, and  $g_{3}(L)=140
G_{6}(L)$ correspond to the lattice functions $G_{2n}(L)= \sum_{w \in
  L} \frac{1}{w^{2n}}$ (see Section $6.1$ of \cite{Shimura} for
example). These functions satisfy $f_{a}(\gamma(z))=f_{a\gamma}(z)$
for every $\gamma \in \SL_{2}(\ZZ) $. Let $\mathcal{R}_p$ be the field
of modular functions of level $p$, which by Proposition $6.1$ of
\cite{Shimura} is
\[   
\mathcal{R}_{p}=\CC(j,f_{a}  \mid  a \in (p^{-1} {\ZZ}^2)/{\ZZ^2},  a \notin {\ZZ}^{2}   ). 
\]
Let $\xi_{p}$ be  a fixed p-th
root of unity and let $\sigma \in \Gal(\CC/\QQ(\xi_{p}))$. Since the functions $j,f_{a}$
have Fourier expansions belonging to $\QQ(\xi_{p})$, if
$f= cj+ \sum_{a}c_{a} f_{a}$, then 
$\sigma(f)= \sigma(c)j+ \sum_{a}\sigma(c_{a}) f_{a}.$
If we choose representatives $\left\{\beta_{k} \right\}$ for $\pm \Gamma(p) \backslash
\Gamma^{\varepsilon}_{ns}(p)$, 
the field of
modular functions for the non-split Cartan is the subfield of $\mathcal{R}_p$ given by
\[ 
\mathcal{R}^{\varepsilon}_{ns}(p)=\CC(j,\sum_{i}f_{a\beta_{i}}).
\] 
Clearly it does not depend on the representatives chosen. In order to
understand the action of $\Gal(\CC/\QQ)$ on modular forms for the
Cartan non-split group, it is enough to understand the effect of
$\Gal(\QQ(\xi_{p}) \slash \QQ)$ on them. For every $n$ relatively prime to $p$ consider the automorphism $\sigma_n$ given by
$\sigma_n(\xi_{p})={\xi_{p}}^{n^{-1}}$. This Galois automorphism depends only on the class of $n$ modulo $p$.
By theorem $6.6$ of \cite{Shimura} and
Theorem $3$ (Chapter $6$, section $3$) of \cite{Lang} this
automorphism acting on the meromorphic modular functions $f_{a}$ is
given by $f_{a \alpha_{{n}^{-1}}}$.

\begin{prop}  \label{prop:galoismeromprphic}
  Let $f$ be a meromorphic form of weight $0$ for
  $\Gamma_{ns}^{\varepsilon}(p)$. Let $\sigma \in \Gal(\CC/\QQ)$ satisfying $\sigma \mid_{\QQ(\xi_{p})}= \sigma_{n}$. Then
  $\sigma(f)$ is a meromorphic form of weight $0$ for
  $\Gamma_{ns}^{\varepsilon n^2}(p)$.
\end{prop}
\begin{proof}
Choose representatives $\left\{\beta_{k} \right\}$ for $\pm \Gamma(p) \backslash
\Gamma^{\varepsilon}_{ns}(p)$ 
such that the $(1,2)$ entries of the matrices are divisible by $n$. 
Since $f$ is a meromorphic form of weight $0$, we have
 \[ f=\lambda j+ \sum_{a} \lambda_a \sum_{i} f_{a \beta_{i}}.  \]
Then, 
\[\sigma(f)=\sigma(\lambda)j+ \sum_{a} \sigma(\lambda_a)\sum_{i} f_{a \beta_i \alpha_{n^{-1}}}.\] 
The action on each $f_{a}$ can be written
as $f_{a \beta_{i} \alpha_{n^{-1}}  } = f_{a \alpha_{n^{-1}}(\alpha_{n} \beta_{i} \alpha_{n^{-1}})}$. 
  Since $\left\{\alpha_{n} \beta_{k} \alpha_{n^{-1}} \right\}_{k}$ are representatives for $\pm \Gamma(p) \backslash   \Gamma_{ns}^{\varepsilon
  n^2}(p)$ we see that $\sigma(f)$ is an automorphic form for the required group.
\end{proof}

\begin{rem} \label{rem:weight0} Although the last result is only
  stated for weight $0$ forms, it also applies to modular forms of
  other weights by dividing the form by an appropriate Eisenstein series
  with rational Fourier coefficients.
\end{rem}

\begin{prop} Let $f$ be a meromorphic modular function for
  $\Gamma^{\varepsilon}_{ns}(p)$ and let  $\sigma \in \Gal(\CC/\QQ)$ satisfying $\sigma \mid_{\QQ(\xi_{p})}= \sigma_{n}$.
  Then $\sigma(\upsilon^{\varepsilon}_{m}(f))=\upsilon^{\varepsilon n^2}_{m}(\sigma(f))$.
\end{prop}

\begin{proof}
  Choose $A^{\varepsilon}_{m}$ in such a way that its $(1,2)$ entry
  is divisible by $n$.  It is easy to see that
  $\left(\begin{smallmatrix} 1 & 0 \\ 0 & n\end{smallmatrix} \right)
  A^{\varepsilon}_{m} \left(\begin{smallmatrix} 1 & 0 \\ 0 &
      1/n \end{smallmatrix} \right)$, which belongs to $\SL_{2}(\ZZ)$
  by our choice of $A^{\varepsilon}_{m}$, gives the same action on
  the $f_{a}$ as $A^{\varepsilon n^2}_{m}$ (since both matrices are
  easily seen to be equivalent modulo $p$). This proves the result on
  weight zero forms. For general weights, the same argument as in Remark
  \ref{rem:weight0} applies.
\end{proof}

\begin{coro}
  With the previous notation, $\Hecke^{\varepsilon
    n^2}_{m}(\sigma(f))=\sigma(\Hecke^{\varepsilon}_{m}(f))$.
\label{coro:sigmaaction}
\end{coro}
\begin{proof}
  This follows from the previous Proposition and the fact that $\sigma$
  commutes with $T_m$ (this is easily obtained by looking at
  the action on $q$-expansions).
\end{proof}

\begin{coro}
 Suppose that $g \in S_2(\Gamma_0(p^2))$ is a newform  with rational eigenvalues. Then $\sigma(g_{\varepsilon}) \in S_2(\Gamma_{ns}^{\varepsilon n^2}(p))  $ is a normalized newform
 with the same eigenvalues as $g_{\varepsilon}$, i.e. if
$\Hecke^{\varepsilon}_{m}(g_{\varepsilon})= \lambda_{m} g_{\varepsilon}$ with $\lambda_{m} \in \QQ$,
 then
$\Hecke^{\varepsilon n^{2}}_{m}(\sigma(g_{\varepsilon}  ))  = \lambda_{m} \sigma(g_{\varepsilon})$.
\end{coro}

\begin{coro}
 With the previous notations, if $m$ is relatively prime to $p$ and satisfies $mn \equiv 1 \bmod{p}$ then there exists $c_{m} \in \CC$ such that $T_{m} g_{\varepsilon} = c_{m}
   \sigma(g_{\varepsilon})$.
 \label{coro:sigma-hecke}
 \end{coro}

 \begin{proof}
   By Proposition~ \ref{prop:heckeaction}, $T_m(g_{\varepsilon})$ is an eigenform in
   $S_2\left(\Gamma_{ns}^{\varepsilon n^2}(p)\right)$ with the same
   eigenvalues as $g_{\varepsilon}$. By Corollary~\ref{coro:sigmaaction},
   $\sigma(g_{\varepsilon})$ is an eigenform whose eigenvalues
   are the same as those from $g_{\varepsilon}$. The result now follows from
   multiplicity one.
 \end{proof}

\begin{thm} Let $g_{\varepsilon} \in S_2(\Gamma^{\varepsilon}_{ns}(p))$ be a normalized eigenform
  which has the same eigenvalues as a rational newform $g \in S_2(\Gamma_0(p^2))$. Then  $g_{\varepsilon}$
  has a $q$-expansion belonging to $\QQ(\xi_{p})$.
\label{thm:coefficientfield}
\end{thm}

\begin{proof}
  Let $\ell \equiv 1 \bmod{p}$ be such that
  $\lambda_{\ell} \neq 0$, and let $\sigma \in \Gal(\CC/\QQ(\xi_{p}))$ be
  arbitrary. By Corollary~\ref{coro:sigma-hecke} , there is a $c_{\ell}$ such that
\[
T_{\ell} g_{\varepsilon}=c_{\ell} \sigma (g_{\varepsilon}).
\] 
We know that $T_{\ell}g_{\varepsilon}=\lambda_{\ell}g_{\varepsilon}$ (by Corollary~\ref{coro:sameHecke}).  Looking at the first
Fourier coefficient, we get that $c_{\ell} =
\lambda_{\ell}$ and hence $g_{\varepsilon}=\sigma(g_{\varepsilon})$.  Since
$\sigma \in \Gal(\CC/\QQ(\xi_{p}))$ is arbitrary it
follows that the $q$-expansion of $g_{\varepsilon}$ lies in the desired extension.
\end{proof}

\subsection{Rational modular forms}
The curve $X^{\varepsilon}_{ns}(p)$ is defined over $\QQ$ and has $(p-1)$ cusps, all of them defined over $\QQ(\xi_{p})$ and conjugate by $\Gal(\QQ(\xi_{p})/\QQ)$
(see \cite{Serre} Appendix 5).  If $\sigma_{n} \in
\Gal(\QQ(\xi_{p})/\QQ)$, then there exists $A \in \SL_2(\ZZ)$ such
that $\sigma_{n}(\infty) = A \infty$. The matrix $A$ can be
taken to be equal to $A^{\varepsilon}_{n}$ as defined before
Lemma~\ref{lemma:matA}. Recall that if $f$ is a weight $k$ modular
form, its Fourier expansion at the cusp $A^{\varepsilon}_{n} \infty$ is
given by the Fourier expansion of the form
$f|_k[(A^{\varepsilon}_{n})^{-1}]$ at the infinity cusp.

Let $\mathcal{F}^{\varepsilon}_{ns}(p)$ be the field of rational
meromorphic functions for the Cartan non-split group
$\Gamma_{ns}^\varepsilon(p)$,
i.e. $\mathcal{F}^{\varepsilon}_{ns}(p):=\QQ(j,\sum_{i}f_{a\beta_{i}})$.
Combining Proposition~\ref{prop:galoismeromprphic} with
Lemma~\ref{lemma:matA}, it is easy to see that
$\mathcal{F}^{\varepsilon}_{ns}(p)$ consists of all meromorphic functions
invariant for $\Gamma^{\varepsilon}_{ns}(p)$, whose $q$-expansions at infinity
belong to $\QQ(\xi_{p})$ and such that the Fourier expansion at
$\sigma_{n}(\infty)$ equals $\sigma_{{n}^{-1}}(f)$. As in
Remark~\ref{rem:weight0}, the same argument applies to other weights.

\begin{definition}[Rational Modular Forms]
  A form $f \in S_2(\Gamma_{ns}^\varepsilon(p))$ is called
  \emph{rational} if its $q$-expansion at every cusp belongs to $\QQ(\xi_{p})$ and the expansion at the cusp
  $\sigma_{n}(\infty)$ equals that of
  $\sigma_{{n}^{-1}}(f)$ at the infinity cusp for all $n$ relatively prime to $p$.
  \label{rationalmodform}
\end{definition}

Recall that if $X$ is a curve defined over a field $K$, a
  differential form defined over $K$ is a differential form which is
  locally of the form $f dg$, where $f$ and $g$ are meromorphic
  forms defined over $K$.

  \begin{prop}
    If $f \in S_{2}(\Gamma_{ns}^\varepsilon(p))$ is rational, it
    defines a rational meromorphic differential form $f(q)
    \frac{dq}{q}$ on $X_{ns}^\varepsilon(p)$, where $q=e^{\frac{2 \pi
        i z}{p}}$.
 \label{prop:differentialQ}
  \end{prop}
  \begin{proof}
 Note
  that
\[
f(q) \frac{dq}{q}= \frac{2 \pi i}{p} f(z) dz= \frac{f(z)}{ \frac{p j'(z)}{2\pi i}}dj .
\]
Since $j$ belongs to $\mathcal{F}^{\varepsilon}_{ns}(p)$ and $\frac{p
  j'}{2 \pi i}$ is a rational meromorphic function with respect to
$\SL_{2}(\ZZ)$ (of weight two) , their quotient lies in
$\mathcal{F}^{\varepsilon}_{ns}(p)$ as claimed.
\end{proof}

Theorem~\ref{thm:coefficientfield} says that $g_{\varepsilon}$ has $q$-expansion
with coefficients in $\QQ(\xi_p)$. If we multiply the form by any constant
in such field, the same holds. What
is the right way to normalize $g_{\varepsilon}$?

\begin{thm}
\label{thm:normalization}
Let $g_{\varepsilon} \in S_2(\Gamma_{ns}^\varepsilon(p))$ be an
eigenform with rational eigenvalues. Then there exists a constant
$c \in \QQ(\xi_p)$ such that $cg_{\varepsilon}$ is rational. Such
constant is unique up to multiplication by a non-zero rational
number.
\end{thm}
\begin{proof} 
  It is clear that $c$, if exists, is unique up to multiplication by a non-zero rational number. By Proposition~\ref{prop:heckedecomposition}, it is enough to find $c \in \QQ(\xi_{p})$
  such that that for all prime numbers $\ell$
\[ 
T_{\ell}(c g_{\varepsilon})=\lambda_{\ell} \sigma_{{\ell}^{-1}}(c g_{\varepsilon}).
\]

  We have that for each
  $\ell$, there exists $c_{\ell} \in \QQ(\xi_{p})$, which only depends on the class of $\ell$ modulo $p$, such that
  $T_{\ell} g_{\varepsilon}=\lambda_{\ell} c_{\ell} \sigma_{{\ell}^{-1}}(g_{\varepsilon})$ .  We
  need to find a non-zero $c \in \QQ(\xi_{p})$ such that
  $T_{\ell}(c g_{\varepsilon})=\lambda_{\ell} \sigma_{{\ell}^{-1}}(c g_{\varepsilon})$ , i.e. $c_{\ell}=\sigma_{\ell^{-1}}(c)/ c$.

  Let $\ell$ be such that its class modulo $p$ is a generator of
  $\mathbb{F}^{*}_{p}$ and let
  $\left\lbrace \ell_{i} \right\rbrace _{1 \le i \le p-1}$ be distinct
  primes in the same class of $\ell$ modulo $p$ such that
  $\lambda_{\ell_{i}} \neq 0$ (since $g$ does not have complex
  multiplication, such primes exist by Serre's open image Theorem or
  Sato-Tate Theorem).  In that case,
  $\prod_{i=1}^{p-1} \ell_i \equiv 1 \bmod p$ and
\[ 
(\Pi \, \lambda_{\ell_{i}}) g_{\varepsilon}= \Hecke_{\Pi \, \ell_i}(g_{\varepsilon}) = T_{\Pi \, \ell_{i}}(g_{\varepsilon})= T_{\ell_{1}}\circ \dots \circ T_{\ell_{p-1}}(g_{\varepsilon})= (\Pi \, \lambda_{\ell_i}) \Nm^{\QQ(\xi_{p})}_{\QQ} (c_{\ell}) g_{\varepsilon}.
\]
Since $\Nm^{\QQ(\xi_{p})}_{\QQ} (c_{\ell})=1$, by Hilbert theorem
$90$ there exists $c \in \QQ(\xi_{p})$ that satisfies
$c_{\ell}=\sigma_{\ell^{-1}}(c)/ c$. Since $\ell$ is a generator of
$\mathbb{F}^{*}_{p}$ it is easy to see that $c$ satisfies
$c_{q}=\sigma_{q^{-1}}(c)/ c$ for every $q$ relatively prime to $p$.
\end{proof}

\begin{rem} Let $g_{\varepsilon} \in S_{2}(\Gamma^{\varepsilon +}_{ns}(p))$. If $\ell
  \equiv -1 \bmod{p}$, $\sigma_{\ell}$ corresponds to complex
  conjugation in $\QQ(\xi_p)$. Since the characters involved in the
  sum are even characters, $\chi(\ell)=1$, and by the last Proposition
  $\sigma_\ell$ acts trivially. This implies that the coefficients of
  the modular forms in fact lie in
  $\QQ(\xi_{p}+{\xi_{p}}^{-1})=\QQ(\xi^{+}_{p}) $. Similarly, if $g_{\varepsilon}
  \in S_{2}(\Gamma^{\varepsilon -}_{ns}(p))$, the coefficients will be purely
  imaginary.
\label{rem:sign}
\end{rem}

Note that even for a rational modular form, it is not clear how to
choose the rational multiple of it which should correspond to
``$a_1=1$'' in the classical case. The best one can do is to choose the
coefficients to be algebraic integers and have no common rational integer factor.

\begin{definition}
The \emph{proper normalization} of $g_{\varepsilon}$ is the unique (up to sign) renormalization $G_{\varepsilon}$ of $g_{\varepsilon}$ that satisfies:

  \begin{itemize}
  \item  $G_{\varepsilon}$ is a rational newform.
  \item The Fourier expansion of $G_{\varepsilon}$ has algebraic integer coefficients.
  \item If $n\in \ZZ$ and $n\ge 2$, $\frac{G_{\varepsilon}}{n}$ does not have
    integral coefficients.
  \end{itemize}
\end{definition}

\begin{rem}
 If $G_{\varepsilon} \in S_2(\Gamma_{ns}^\varepsilon(p))$ is a properly-normalized eigenform with rational eigenvalues then $\sigma_{n} (G_{\varepsilon}) \in S_2(\Gamma_{ns}^{\varepsilon n^2}(p))$ is a properly-normalized eigenform with rational eigenvalues. Moreover since $G_{\varepsilon}$ is rational, we must have $\sigma_{n}(G_{\varepsilon})=G_{\varepsilon}|_k[(A^{\varepsilon}_{n})]$ (see Definition \ref{rationalmodform}).

\label{rem:changeofcartan}
\end{rem}

 \subsection{ Eichler-Shimura}
 The Eichler-Shimura construction (Theorem 7.9 of \cite{Shimura})
 associates to $G_{\varepsilon}$ the abelian variety
 $\A_{G_{\varepsilon}}:= \Jac(X_{ns}^\varepsilon(p))/
 (I_{G_{\varepsilon}} \Jac(X_{ns}^\varepsilon(p)))$,
 where $I_{G_{\varepsilon}}$ is the kernel of the morphism from
$R(\Gamma^{\varepsilon}_{ns}(p),\Delta^{\varepsilon}_{ns}(p)) \rightarrow \mathbb{Z}$ which is
 given by sending $\Hecke^{\varepsilon}_{n}$ to the eigenvalue
 $\lambda_{n}$. We have the diagram
\[
\xymatrix{
X_{ns}^\varepsilon(p)\ar@{^{(}->}[r]^{i} \ar@{..>}[dr] & \Jac(X_{ns}^\varepsilon(p))\ar[d]\\
 & \A_{G_{\varepsilon}}
}
\]
where $i$ is the map sending $P$ to
$\left(P \right) - \left(\infty \right)$ and the vertical map (which
is clearly rational) is given by the classical Abel-Jacobi map
given by integrating the differential form
$G_{\varepsilon}(q) \frac{dq}{q}$ and its Galois conjugates over
cycles. By Proposition \ref{prop:differentialQ} this differential is rational, thus the
abelian variety $\A_{G_{\varepsilon}}$ is of dimension $1$, and by
Theorem ~\ref{thm:chen} isogenous to the strong Weil curve $E_g$
attached to $g$. The elliptic curve $\A_{G_{\varepsilon}}$ will be
called the optimal quotient of $\Jac(X^{\varepsilon}_{ns}(p))$
(note that it might not be isomorphic to $E_g$).

Since the cusps of the Cartan curve are defined over $\QQ(\xi_{p})$
(and are Galois conjugates over that field) the map $i$ will not be
defined over $\QQ$. Nevertheless, we can solve this problem by
averaging over all the conjugates of this map; that is, we consider
the following diagram
\[
\xymatrix{
X_{ns}^\varepsilon(p)\ar@{^{(}->}[r]^{\iota} \ar@{..>}[dr]_{\Phi_p^\varepsilon} & \Jac(X_{ns}^\varepsilon(p))\ar[d]\\
 & \A_{G_{\varepsilon}}
}
\]
where $\iota$ is the map sending $P$ to
$ \sum_{\sigma \in \Gal(\QQ(\xi_{p})/\QQ)} (P)-(\sigma(\infty))$. This
is the right and natural definition to make a map defined over $\QQ$
out of $i$.  Therefore, the dot map (that we still call \emph{modular
  parametrization}) is defined over $\QQ$.

\begin{rem} If $G_{\varepsilon} \in S_{2}(\Gamma^{\varepsilon +}_{ns}(p))$, since the normalizer has $(p-1)/2$ cusps, all defined and conjugate over the maximal real subfield of $\QQ(\xi_{p})$, we will take the average in the definition of  $\iota$ over all such cusps.
\end{rem}

\begin{lemma}
  Let $n$ be relatively prime to $p$. Then $\A_{G_{\varepsilon}} =
  \A_{G_{\varepsilon n^2}}$.
\end{lemma}
\begin{proof}
  It is enough to see that the lattice of periods of $G_{\varepsilon}$ is the same
  as the lattice of periods of $\sigma_{n}(G_{\varepsilon})=G_{\varepsilon n^2}$ which is a rational
  eigenform for $S_2(\Gamma_{ns}^{\varepsilon n^2}(p))$ (Remark
  ~\ref{rem:changeofcartan}).  Let $D$ be the closed cycle
  $\left\lbrace \tau , M^{\varepsilon} \tau \right\rbrace$ with
  $M^{\varepsilon} \in \Gamma^{\varepsilon}_{ns}(p)$. Integrating $G_{\varepsilon}$
  over that cycle we get
 \[ \int_{ \tau }^ {M^{\varepsilon} \tau} G_{\varepsilon}(q) \frac{dq}{q} . \]
 By changing variables $z \mapsto \  [A^{\varepsilon}_{n}]^{-1} z$ we obtain
 
 \[ \int_{ [A^{\varepsilon}_{n}]^{-1}  \tau}^ { [A^{\varepsilon}_{n}]^{-1}   M^{\varepsilon}  \tau }  G_{\varepsilon}|_k[(A^{\varepsilon}_{n})]         \frac{dq}{q}
  = \int_{ [A^{\varepsilon}_{n}]^{-1}  \tau}^ { [A^{\varepsilon}_{n}]^{-1}   M^{\varepsilon} [A^{\varepsilon}_{n}] [A^{\varepsilon}_{n}]^{-1}     \tau } \sigma_{n}(G_{\varepsilon})       \frac{dq}{q}.
 \]
 This expression is the integral of $\sigma_{n}(G_{\varepsilon})$ over
 the cycle
 $\left\lbrace \tau' , [A^{\varepsilon}_{n}]^{-1} M^{\varepsilon}
   [A^{\varepsilon}_{n}] \tau' \right\rbrace$,
 where $\tau'=[A^{\varepsilon}_{n}]^{-1} \tau$. Since
 $[A^{\varepsilon}_{n}]^{-1} M^{\varepsilon} [A^{\varepsilon}_{n}] \in
 \Gamma^{\varepsilon n^2}_{ns}(p)$,
 it gives a closed cycle on $\Jac(X^{\varepsilon n^2}_{ns}(p))$.
\end{proof}

Let $E$ denote the elliptic curve $\A_{G_{\varepsilon}}$ (which does
not depend on $\varepsilon$). If $\omega_{E}$ is a holomorphic
differential on $\CC/{\Lambda_{E}}$ its pullback under
$\Phi_p^\varepsilon$ is a constant multiple of
$G_{\varepsilon}(q) \frac{dq}{q}$ (by multiplicity one), where $q=e^{\frac{2 \pi i
    z}{p}}$.
Such constant will be called the Manin constant
$c_{\varepsilon}$. Since $E, \Phi_p^\varepsilon $ and
$G_{\varepsilon}(q) \frac{dq}{q}$ are rational, the Manin constant
must be a rational number.  It is not difficult to see that the
Manin constant does not depend on $\varepsilon$ so we can speak of the
Manin constant $c$.

\begin{prop}
  Let $\Lambda_{G_{\varepsilon}}$ be the lattice attached to $G_{\varepsilon}$ and $c$ the Manin
  constant. Let $\Phi_\omega : \CC/\Lambda_{G_{\varepsilon}} \to E$ be the Weierstrass
  uniformization. Then $\Phi_p^\varepsilon
  (\tau)= \Phi_{\omega}(z_{\tau})$, where  
 \[ z_{\tau}= c \left( \frac{2\pi i}{p}  \left(\sum_{\sigma_{n}
  \in \Gal(\QQ(\xi_{p})/\QQ)} \int_{\infty}^{A_{n^{-1}}\tau}
\sigma_{n}(G_{\varepsilon}^{\varepsilon})(z)dz\right) \right)  \]
 \label{eq:modparam}
\end{prop}

\begin{proof}
 This follows from Proposition 2.11 of \cite{Dar04} and the identity
 
 \[   \int_{ \sigma_{n}(\infty) }^ \tau G_{\varepsilon}(q) \frac{dq}{q}=\int_{ \infty }^
{A_{n}^{-1} \tau} G_{\varepsilon}|_2[A_{n}] (q) \frac{dq}{q} = \int_{\infty}^ {A_{n}^{-1} \tau} \sigma_{n}(G_{\varepsilon} (q)) \frac{dq}{q}  .  \]

\end{proof}

\section{ General levels}
\label{section:generallevels}
In this section we generalize the previous results to more general
conductors. Thanks to the Chinese Reminder Theorem, the theory works
exactly the same as in the $p^2$ case. Let $E/ \QQ$ be an elliptic
curve of conductor $N^2m$ with $\gcd(N,m)=1$, and $N=p_1 \dots p_r$
($p_{i}$ distinct odd primes). By Shimura-Taniyama-Wiles, there exists
an eigenform $g \in S^{new}_{2}(\Gamma_{0}(N^2 m))$ with rational
eigenvalues whose attached elliptic curve is $E$.  Let $\varepsilon_i$
be a non-square modulo $p_i$, for $i=1,\dots ,r$ and let
$\vec{\varepsilon}=(\varepsilon_1,\ldots,\varepsilon_r)$. Let
$\Gamma_{ns}^{\vec{\varepsilon}}(N,m)=
\cap_{i=1}^r\Gamma^{\varepsilon_i}_{ns}(p_i) \cap \Gamma_{0}(m)$
and consider the curve
$X_{ns}^{\vec{\varepsilon}}(N,m)=\Gamma_{ns}^{\vec{\varepsilon}}(N,m)
\backslash \mathcal{H} ^{*}$.

The moduli interpretation is a mix of the classical one and the one
of the previous section. We consider tuples
$(E,\psi,\phi_1,\ldots,\phi_r)$, where $E/\CC$ is an elliptic curve,
$\psi:E \to E'$ is a cyclic degree $m$ isogeny (or equivalently a
cyclic subgroup of order $m$), and $\phi_i \in \End_{\FF_{p_{i}}}(E[p_i])$
is such that $\phi_i^2$ corresponds to multiplication by
$\varepsilon_i$ for $i=1,\ldots,r$. A computation similar to that of
Proposition~\ref{prop:moduliinterpretation} shows that
$X_{ns}^{\vec{\varepsilon}}(N,m)$ represents the moduli problem
stated.

\medskip

We have the following generalization of Theorem \ref{thm:chen}.
\begin{thm}
  $\Jac(X_{ns}^{\vec{\varepsilon}}(N,m))$ is isogenous over $\QQ$ to
  $\Jac(X_{0}(N^2m))^{N^2\text{-new}}$ by a Hecke equivariant map.
\end{thm}

\begin{proof} Let $X(Nm)$ be the modular curve which is the
  compactified moduli space of triples $(E/S/\QQ,\phi)$, where $S$ is
  a $\QQ$ scheme, $E/S$ is an elliptic curve and $\phi:(\ZZ/Nm)^2_S
  \mapsto E[Nm]$ is an isomorphism of group schemes over $S$. The
  group $\GL_2(\ZZ/Nm)$ acts on the right on $X(Nm)$. If $\Gamma$ is
  any subgroup of $\GL_2(\ZZ/Nm)$, one can consider the quotient
  $X(Nm)/\Gamma$ via an appropriate moduli interpretation. We are
  interested in the following two subgroups (as subgroups of
  $\GL_2(\ZZ/Nm)$): $\Gamma_{ns}^{\vec{\varepsilon}}(N,m)$ and
  $\tilde{\Gamma}:= \cap_{i=1}^r T(p_i) \cap \Gamma_0(m)$, where
  $T(p)$ is the standard maximal torus modulo $p$ (consisting of
  diagonal matrices). The quotients correspond respectively to
  $X_{ns}^{\vec{\varepsilon}}(N,m)$ and $X_0(N^2m)$ (as in \cite{Edi},
  (1.0.4)).

  Using an inductive argument, it is enough to prove that the Jacobian
  of the quotient by $\Gamma_1=\cap_{i=1}^r T(p_i) \cap \Gamma_0(m)$
  is isomorphic to the $p_1$-new part of the quotient by
  $\Gamma_2=\cap_{i=2}^r T(p_i) \cap \Gamma_0(p_1^2m)$. But in this
  case, one can prove Proposition 1.2 of \cite{Edi} in exactly the
  same way, where now the subgroups of such paper correspond to the
  local components at $p_1$ of our subgroups (since both groups are
  the same at all the other primes). Then, the same formalism as
  Theorem 1.3 (of \cite{Edi}) proves our claim.
\end{proof}

The previous theorem, together with the comments in the proof of
Theorem~\ref{thm:chen}, imply that there exists $g_{\vec{\varepsilon}} \in
S_{2}(\Gamma_{ns}^{\vec{\varepsilon}}(N,m))$ with the same eigenvalues
for the Hecke operators $\Hecke^{\vec{\varepsilon}}_{n}$ as $g$ outside the primes
$p_{i}$. The theory works the same as in the level $p^2$ case, with
some minor changes. 

The geometric definition of Hecke operators is the same as before. We
consider all degree $n$ cyclic isogenies (for $n$ prime to $Nm$) and consider
the same action on each $\phi_i$ and, as in the classical case, the
image of the cyclic subgroup by our isogeny.

The algebraic definition is also the same, and the operator
$\upsilon^{\vec{\varepsilon}}_n$, as well as coset representatives,
are defined via a matrix $A^{\vec{\varepsilon}}_{n} \in \Gamma_{0}(m)$
which satisfies the corresponding congruence modulo all the prime
numbers $p_i$.

Note that $\sigma_{{n}^{-1}}$ and $T_{n}$ will
send modular forms for $\Gamma_{ns}^{\vec{\varepsilon}}(N,m)$ to
modular forms for $\Gamma_{ns}^{\vec{\varepsilon}/{n}^{2}}(N,m)$,
 and all the results from the previous Section generalize trivially.
 In particular, we have the analogue of Theorem~\ref{thm:combination}.
	
\begin{thm}
  Let $g_{\vec{\varepsilon}} \in S_2(\Gamma_{ns}^{\vec{\varepsilon}}(N,m))$ be an
  eigenform. Then there exists  eigenforms $h_{i} \in
  S_{2}(\Gamma_{0}(N_{i}m),\chi_i)$, with $N_{i}\mid N^2$, and $\chi_{i}$ a
   character modulo $N^2/N_{i}$ such that
\[
\widetilde{g_{\vec{\varepsilon}}} = \sum_{\chi} a_{\chi} (g \otimes \chi)+  \sum_{i} a^{1}_{i}h_{i}+ a^{2}_{i}\bar{h_{i}},
\]
where the first sum is over all characters modulo $N$.
\label{thm:combinationgeneral}
\end{thm}

\begin{proof}
  We need to look at the local representations of our form $g$. Let
  $p$ be a prime dividing $N$. If $g$ is supercuspidal at $p$, then
  all of its twists by characters of conductor $p$ have the same level
  as $g$ (possibly with a character). If $g$ is a ramified principal
  series or Steinberg at $p$, then there exists a character $\chi_p$
  such that
  $g \otimes \chi_p \in
  S_2\left(\Gamma_0\left(\frac{N^2}{p}m,\chi_p^2\right)\right)$
  (note that this is true locally, but since the class number of $\QQ$
  is one, and there are no units, all local characters can be extended
  to global characters). We take $\chi_i= \prod_p \chi_p$ and $h_{i}$
  the new form attached to $g \otimes \chi_i$. Note that for each
  prime at which the representation is a ramified Principal Series we
  might have two choices of the character ($\chi_{p}$ ad
  $\bar{\chi_{p}}$) giving us also
  $\bar{\chi_i}= \prod_p \bar{\chi_p}$ and $\bar{h_{i}}$ the newform
  associated to $g \otimes \bar {\chi_{i}}$.  Now the same proof as in
  Theorem~\ref{thm:combination} applies.
\end{proof}

Using this Theorem we can also compute the Fourier expansion and define $G_{\vec{\varepsilon}}$ as a proper-normalization of $g_{\vec{\varepsilon}}$.  Now the coefficient field will be
$\QQ(\xi_{p_1},\ldots,\xi_{p_r})$, whose Galois group is isomorphic to
$\prod_i \FF_{p_i}^\times$ and the modular parametrization $\Phi_{N}^{\vec{\varepsilon}}$  map can be written in the form $\Phi_{\omega}(z_{\tau})$ where
\begin{equation}
  \label{eq:modularparamgeneral}
z_{\tau} = c \frac{2\pi i }{N}  \sum_{\sigma \in
  \Gal(\QQ(\xi_{N})/\QQ)} \int_{\infty}^{A_\sigma^{-1}\tau}
\sigma(G_{\vec{\varepsilon}})(z)dz.  
\end{equation}
Using the Fourier expansion of $G_{\vec{\varepsilon}}$, we can calculate the integral
numerically to arbitrary precision. Recall that the convergence of
such integral is exponential depending on the imaginary part of the
point on the upper half plane.

Summing up, we have obtained a modular parametrization
\begin{equation}
  \Phi^{\vec{\varepsilon}}_N: X_{ns}^{\vec{\varepsilon}}(N,m) \rightarrow E(\mathbb{C}) 
  \label{eq:modularparam}
\end{equation}
defined over $\mathbb{Q}$.
  We make the following observation about the Manin constant, which is supported by the evidence shown in the examples.
\begin{conj} 
The Manin constant belongs to $\ZZ [1/N]$.
\end{conj}
This conjecture should follow from similar arguments as exposed in \cite{Maz78}.

\section{Heegner points on general Cartan non-split curves}
Let $E/\QQ$ be an elliptic curve and let $\Om=\<1,\omega>$ be an order
in an imaginary quadratic field $K$.  We say that the pair
$\left( E, \Om \right)$ satisfies the \emph{Cartan-Heegner
  hypothesis} if the following holds:

\begin{itemize}
\item The conductor of $E$ is $N^2m$ where $N$ is an odd square-free number and $\gcd(N,m)=1$.
\item The discriminant $d$ of $\Om$ is prime to $Nm$.
\item Every prime dividing $m$ is split in $\Om$.
\item Every prime dividing  $N$ is inert in $\Om$.
\end{itemize}

Note that $\Om$ satisfies the classical Heegner hypothesis at the
primes dividing $m$ but not at the primes dividing $N$, therefore, we will not be able to construct Heegner points on $X_0(N^2m)$.
Given a pair $(E,\Om)$ satisfying the Cartan-Heegner hypothesis we will use the
letters $N$ and $m$ to denote the factorization of the conductor of $E$ as in the definition. 

Recall that a matrix $M \in M_{2\times 2}(\ZZ)$ with
$\Tr(M)=\Tr(\omega)$ and $\det(M) = \Nm(\omega)$ gives an embedding
$\Om \hookrightarrow M_{2\times2}(\ZZ)$ given by sending $\omega$ to $M$.
A Heegner point on $X_{ns}^{\vec{\varepsilon}}(N,m)$ with endomorphism ring $\Om$ is a point
$\tau$ on the upper half plane which is fixed by a matrix $M \in M^{\vec{\varepsilon}}_{ns}(N) \cap M_{0}(m)$ satisfying the above conditions.

Let $H$ be the Hilbert class field of $\Om$. To a Heegner point $\tau$
one associates the elliptic curve
$E_{\tau} = \mathbb{C} /\left\langle 1, \tau\right\rangle$. The fact that $\tau$ is fixed by $M$ allows to associate to $\tau$ a pair of points in
$X^{\vec{\varepsilon}}_{ns}(N,m)(H)$ conjugate under
$\Gal(H/ \mathbb{Q}(j(E_{\tau})))$ (see \cite{Serre} Appendix $5$ for more
details).  A Heegner point on $E$ with endomorphism ring $\Om$ is
the image of a Heegner point with endomorphism ring $\Om$ in
$X^{\vec{\varepsilon}}_{ns}(N,m)(H)$ under the modular parametrization
(\ref{eq:modularparam}).

\subsection{Moduli interpretation}
In order to construct systems of Heegner points, it is also useful to have
a definition of Heegner points in terms of the moduli
interpretation. 
\begin{definition}
 A Heegner point on $X^{\vec{\varepsilon}}_{ns}(N,m)$ is a tuple $[\Om,[\mathfrak{a}],\id{m},\phi_{\alpha}]$ where $\Om$ is as before, $[\mathfrak{a}]$ is an element in
  $\Pic(\Om)$ which determines an elliptic curve $E_{\mathfrak{a}}=\Om/{\mathfrak{a}}$ with complex multiplication by $\Om$, $\id{m}$ is a cyclic ideal in
$\Om$ of norm $m$  and $\phi_{\alpha} \in \prod_{p \mid N} \End_{\FF_p}(E_{\mathfrak{a}}[p])$ is such that 
\begin{itemize}
\item $\phi_\alpha ^2$ is given by multiplication by $\vec{\varepsilon}$.
  \item There exists $\alpha \in \Om$ such that $\phi_\alpha$ is given by multiplication by $\alpha$ on each coordinate.
 \end{itemize}
 \label{definition:Heegnerpoint}
\end{definition}

\begin{rem}
  The element $\alpha$ is well defined modulo $N$, which is a product of inert primes of $\Om$, so we can just take $\alpha \in \Om/N$.
\end{rem}

 \begin{prop} 
\label{prop:heegnerrelations}
Let $[\Om,[\mathfrak{a}],\id{m},\phi_{\alpha}]$ be a Heegner point.
\begin{enumerate}
  \item If $\tau$ denotes complex conjugation, then $(\Om, [
    \mathfrak{a}],\id{m},\phi_{\alpha})^{\tau}=(\Om, [ {\mathfrak{a}}^{-1}],\overline{\id{m}},\phi_{-\alpha})$
  \item Let $[\mathfrak{b}]$ be a fractional ideal, and let
    $\sigma_{\mathfrak{b}} \in \Gal(H/K)$  be the Artin symbol associated to
    $[\mathfrak{b}]$. Then
   \[ (\mathcal{O}, [ \mathfrak{a}],\id{m},\phi_{\alpha})^{\sigma_{\mathfrak{b}}}=(\mathcal{O}, [ \mathfrak{a}{\mathfrak{b}}^{-1}],\id{m},\phi_{\alpha})
   \]
   \item If $p\mid N$, then $ \omega_{p} (\mathcal{O}, [ \mathfrak{a}],\id{m},\phi_{\alpha})= (\mathcal{O}, [ \mathfrak{a}],\id{m}, \phi_{-\alpha})$.
  \end{enumerate}
\end{prop}

\begin{proof}
  The items $(1)$ and $(2)$ follow from \cite{Serre2} (since $\id{m}$
  and $\alpha$ are defined over $K$), while $(3)$ follows from Remark
  \ref{rem:involution}.
\end{proof}

Using the geometric interpretation of Hecke operators as described in
section~\ref{Heckegeometric} it is clear that we have the following
formula for Hecke operators (for $\ell$ relatively prime to $Nm$) acting on Heegner
points, analogous to the one given in \cite{Gr84} section $6$:

\begin{equation}
  \Hecke^{\vec{\varepsilon}}_{\ell} ( [\Om,\mathfrak{a},\id{m},\phi_{\alpha}] ) =  \sum_{ \mathfrak{a} /\mathfrak{b} \cong \ZZ/ \ell} (End(\mathfrak{b}),\mathfrak{b}, \id{m} \cdot End(\mathfrak{b}) \cap End(\mathfrak{b}) ,\phi_{  \alpha} )  .
\label{eq:heegnerrelations}
\end{equation}

\subsection{Heegner systems}
Fix an elliptic curve $E$ as before, and let $K$ be an imaginary
quadratic field whose maximal order satisfies the \emph{Cartan Heegner
  hypothesis}. Let $n$ be a positive integer prime to
$\Cond(E) \cdot \Disc(K)$. Let $\Om_{n}$ be the unique order in $K$ of
conductor $n$ and let $K_{n}$ be the corresponding Hilbert class
field. The order $\Om_n$ satisfies the \emph{Cartan Heegner
  hypothesis}, so, it gives rise to a set of Heegner points
$HP(n) \subset E(K_{n})$.

\begin{prop}
$1.$ Let $n$ be an integer and let $\ell$ be a prime number, both relatively prime to $\Cond(E) \cdot \Disc(K)$.
Consider any $P_{n \ell} \in HP(n \ell)$. Then, there exists points $P_{n} \in E(H_{n})$ and (when $ \ell \mid n$) $P_{n/\ell} \in HP(n/\ell)$ such that
\begin{itemize}
 \item If $\ell \nmid n$ is inert in $K$, 
\[ 
Tr_{K_{n \ell}/K_{n}} P_{n \ell}= a_{\ell} P_{n},
\]
\item If $\ell=\lambda \bar{\lambda} \nmid n$ is split in $K$, 
\[ 
Tr_{K_{n \ell}/K_{n}} P_{n \ell}= (a_{\ell}-\sigma_{\lambda}- \sigma^{-1}_{\lambda}) P_{n}.
\]

\item If $\ell \mid n$,
\[ 
Tr_{K_{n \ell}/K_{n}} P_{n \ell}= a_{\ell}P_{n}- P_{n/\ell}.
\]
\end{itemize}
where $a_{\ell}=1+\ell- card( \tilde{E}(\mathbb{F}_{\ell}))$. 

$2.$ There exists $\sigma
  \in \Gal(K_n/K)$ such that
\[
{P_{n}}^{\tau} \equiv -\sign(E,\QQ) {P_{n}}^{\sigma} \bmod{E(K_{n})_{tors}},
\]
where $\tau$ is complex conjugation and  $\sign(E,\QQ)$ is the root number of $E/\QQ$.
\label{prop:GZ}
\end{prop}
\begin{proof}
  From Proposition~\ref{prop:heegnerrelations},
  equation~(\ref{eq:heegnerrelations}) and the discussion in between,
  the result follows quite formally. See for example \cite{Gr89}
  Proposition 3.7 and Proposition 5.3 or \cite{Dar04} section 3.4 and
  \cite{GZ} section II.1.
\end{proof}
\begin{definition}
  A \emph{Heegner system} attached to $(E,K)$ is a collection of
  points $P_n \in E(K_n)$ (indexed by positive integers $n$ relatively prime to $\Cond(E) \cdot \Disc(K)$) which satisfies the conditions of the previous Proposition.
  \end{definition}

  If $E$ is a rational elliptic curve and $K$ satisfies the
  Cartan-Heegner hypothesis, Proposition~\ref{prop:GZ}
  proves that the set of Heegner points form a Heegner
  system. Given a Heegner system, Kolyvagin's machinery works and we
  get the following result:
\begin{thm}
 Let $\{P_{n}\}$ be the Heegner system attached to $(E,K)$ as constructed above, where the elliptic curve does not have complex multiplication. Define $P_{K}=Tr_{K_{1}/K} P_{1} \in E(K)$.
If $P_{K}$ is non-torsion then the following are true
 \begin{itemize}
  \item The Mordell-Weil group $E(K)$ is of rank one.
  \item The Shafarevich-Tate group of $E/K$ is finite.
 \end{itemize}
 \label{thm:Kolyvagin}
\end{thm}
\begin{proof}
 See Theorem 10.1 of \cite{Dar04}.
\end{proof}
Furthermore, we have the following crucial relation with L-series derivatives:
\begin{thm}[Gross-Zagier-Zhang]
  The point $P_1$ is non-torsion if and only if $L'(E/K,1) \neq 0$.
  \label{thm:GZZ}
\end{thm}
\begin{proof}
  This is part of Zhang's result in \cite{Zh2}. Note that his choice of order of level $N$ 
  in $(6.3)$ (page 15) coincides with the Cartan
  non-split one. Then, Theorem 6.1 applies, giving a relation between
  the L-series derivative and the Neron-Tate height pairing (inside
  the Jacobian) of the projection of the Heegner point to the
  $f$-isotypical component.
\end{proof}
\begin{rem}
  Zhang's formula is proven for points on the Jacobian of the Cartan
  non-split curve. To get some version of the Birch and
  Swinnerton-Dyer conjecture in this context, the Manin constant and
  the degree of the modular parametrization need to be computed for
  such curve. Unfortunately, no such formulas are known.
\end{rem}
\section{Computational digression}
\subsection{Computing eigenforms}
\label{section: Computational}
Let $g \in S_2(\Gamma_0(N^2m)$ be an eigenform with rational
eigenvalues. We need to compute the Fourier expansion of
$g_{\vec{\varepsilon}}$.
\begin{lemma}
We have $\Gamma_{ns}^{\vec{\varepsilon}}(N,m)/ (\Gamma(N) \cap \Gamma_{0}(m)
  ) \cong \prod_{p \mid N} \ZZ /(p+1) $.
\label{lemma:invariance}
\end{lemma}
\begin{proof}
  The morphism
  $\Gamma^{\varepsilon}_{ns}(p) / \Gamma(p) \to
  \FF_p[\sqrt{\varepsilon}]$
 given by $\left(\begin{smallmatrix} a & b \\
      \varepsilon b & a\end{smallmatrix} \right) \to a+b
  \sqrt{\varepsilon}$
  sends $\Gamma^{\varepsilon}_{ns}(p) / \Gamma(p)$ to
  $\{ \alpha \in \FF_{p^2}^\times\, :\, \Nm(\alpha)=1\}$, which is
  isomorphic to $\ZZ /(p+1)$. The result follows from the Chinese
  Remainder Theorem.
\end{proof}
To compute the Fourier expansion of $g_{\vec{\varepsilon}}$ we proceed as follows:
\begin{enumerate}
\item We compute the local type at each prime dividing $N$. This can
  be done either by looking at the reduced curve and the field where
  it gets semi-stable reduction or by considering twists, as in
  \cite{Pacetti}. Using the local type information, we compute the
  newforms of smaller level that appear in
  Theorem~\ref{thm:combinationgeneral}. If there are some ramified
  principal series primes, one can compute the form $h$ from the
  elliptic curve (see Appendix~\ref{app:1}).
\item Once we have all the forms appearing in
  Theorem~\ref{thm:combinationgeneral}, we are led to compute the
  linear combination. We take a formal linear combination with
  variables $x_i$.  The forms appearing are invariant under
  $\Gamma(N) \cap \Gamma_{0}(m)$, so we have to impose invariance
  under
  $\Gamma_{ns}^{\vec{\varepsilon}}(N,m)/ (\Gamma(N) \cap
  \Gamma_{0}(m))$.
  Using Lemma~\ref{lemma:invariance} we get a set
  $\{\alpha_i\}_{i}$ of generators for the quotient. Imposing invariance under
  $\alpha_i$ (via evaluating the linear combination at some point in
  $\mathcal{H}$) gives a linear equation on the $x_i$'s (with complex
  coefficients). Asking invariance for the whole set of generators, we
  get a linear system, whose solution set $\A$ are the forms in
  $S_2(\Gamma_{ns}^{\vec{\varepsilon}}(N,m))$ with the same
  eigenvalues as $g$ for $n \equiv 1 \pmod N$.
\end{enumerate}
By Theorem \ref{thm:rajan}, $\A$ is the set of twists of the newform $g$ by quadratic characters
 $\chi$ modulo $N$ which are newforms of  level $N^2m$. 
This implies that the space $\A$ has dimension $2^d$, where $d$ is the number of primes dividing $N$ where the local representation is supercuspidal or a principal series (minimal by quadratic twist).
We need to pin
down $g_{\vec{\varepsilon}}$. Here is how to do it. Let $p \mid N$ be
a prime number and let $\varkappa_p$ be the quadratic character modulo
$p$.

\medskip

\noindent {\bf Fact 1:} If $\pi_{p}$ is supercuspidal, let
$\epsilon_p$ denote the local sign at $p$. If $\epsilon_p =1$ then $g$
can be written as a linear combination such as in
Theorem~\ref{thm:combinationgeneral} where only twists of $g$ by characters
with even $p$-part are involved, while for $g \otimes \varkappa_p$
only twists of $g$ by characters with odd $p$-part are involved. If
$\epsilon_p=-1$, the situation is the opposite one.

\noindent {\bf Fact 2:} If $\pi_{p}$ is Principal Series, let $q$ be a
non-square modulo $p$. The operator
$\Hecke^{\varepsilon}_{q}= T_{q} \upsilon^{\varepsilon}_{q}$  acts as $\lambda_{q}$ on the subspace spanned by $g_{\vec{\varepsilon}}$ and as $-\lambda_{q}$ on
the subspace spanned by $({g \otimes \varkappa_{p}})_{\vec{\varepsilon}}$.

\medskip

\noindent{\bf Proof of Fact 1:} Recall from Remark~\ref{rem: evenandodd} that if $\epsilon_p=1$
(resp. $\epsilon_p=-1$) then only twists of $g$ with even $p$-part
(resp. odd $p$-part) are in the sum. By Corollary $3.3$ of
\cite{Pacetti}, the local sign at $p$ changes while twisting $g$ by
$\varkappa_p$ like $-\left(\frac{-1}{p}\right) = -
\varkappa_p(-1)$. Therefore, the variation of the sign at $p$ of the
characters involved in the combination for $g$ and $g \otimes
\varkappa_p$ are different.
\medskip

Each condition halves the dimension and  altogether 
determine $g_{\vec{\varepsilon}}$ up to a constant. Note that the
solution is computed using real arithmetic, so from an approximate
solution we first normalize it such that the first Fourier coefficient
is $1$ (so all coefficients lie in $\QQ(\xi_N)$) and then we
proper-normalize it using an explicit version of Hilbert's 90
Theorem. 
Finally, recall that if  $\gcd(n,N)=1$, the $n$-th coefficient $b_{n}$ of $\widetilde{G_{\vec{\varepsilon}}}$ satisfies
\begin{equation}
\label{eq:coefficientrelation}
b_{n}=\lambda_{n} \sigma_{n^{-1}}(b_{1}).
\end{equation}
Thus, we can obtain the exact Fourier expansion once we have found $b_{1} \in \QQ(\xi_{N})$ and 
 the coefficients at the various $p^{\alpha}_{i}$.
\subsection{Computing Heegner points}
Let $\{\mathfrak{a}_{i}\}$ be a set of representatives of the Class
group of $\Om$ and let $\omega_{i} \in \mathcal{H}$ be such that
$\mathfrak{a}_{i}= \left\langle 1, \omega_{i} \right\rangle$. Let
$M_{\omega_{i}}$ be the set of matrices in $M_{2}(\ZZ)$ that fixes
$\omega_{i}$, which is an order isomorphic to $\Om$. Then,
$M_{\omega_i}$ contains a matrix $N_{i}$ satisfying
$\Tr(M) = \Tr(\omega)$ and $\det(M)=\Nm(\omega)$. 

\medskip

\noindent{\bf Claim:} there exists $A_{i} \in \SL_{2}(\ZZ)$ such that
$A_{i}N_{i} {A_{i}}^{-1} \in M^{\vec{\varepsilon}}_{ns}(N) \cap
M_{0}(m)$.

\smallskip

Then the point $\tau_{i}=A_{i} \omega_{i}$ is a Heegner point on
$X_{ns}^{\vec{\varepsilon}}(N,m)$ with endomorphism ring $\Om$ as wanted.

The matrices $A_i$ are computed in the following way:
\begin{itemize}
\item At a prime $p$ dividing $m$, we chose $A_i^{(p)}$ modulo
  $p^{\val_p(m)}$ of determinant one, taking $N_i$ to an upper
  triangular matrix. This can be done, since the roots of the
  characteristic polynomial of $N_i$ are in $\FF_p$ (since every prime
  that divides $m$ splits in $\Om$), so we just take a basis for the
  Jordan form.
\item At a prime $p$ dividing $N$, since $p$ is inert in $K$, the
  characteristic polynomial of $N_{i}$ is irreducible in
  $\FF_p[x]$. If $N_i = \left(\begin{smallmatrix} \alpha & \beta \\
      \gamma & \delta \end{smallmatrix} \right)$, then we want the
  matrix $A_i$ to satisfy \[
A_i \left(\begin{smallmatrix} \alpha & \beta \\
      \gamma & \delta \end{smallmatrix} \right) = \left(\begin{smallmatrix} \frac{\alpha + \delta}{2} & \sqrt{\frac{d}{\varepsilon}} \\ \varepsilon \sqrt{\frac{d}{\varepsilon}} & \frac{\alpha+ \delta}{2}\end{smallmatrix} \right)A_i \text{ (modulo }p).
\]
We just chose $A_i$ as a matrix in $4$ indeterminates and search for a
non-zero solution of the system (the  determinant of this system is zero, so
there is always such a solution). If the determinant is not $1$, we
just multiply the matrix via an appropriate matrix, as in the proof of
Lemma~\ref{lemma:conjugationrepresentative}.
\end{itemize}

Lastly, the Chinese reminder theorem gives a matrix in
$\SL_2(\ZZ/N^2m\ZZ)$ satisfying our hypotheses, and we lift it to a
matrix in $\SL_{2}(\ZZ)$.

\section{Examples}
In Table \ref{table:curvesdata} we show some examples of our
method. All the examples were done using Pari/GP \cite{PARI}. The table
notation is as follows: the first column is the elliptic curve label
(in Cremona's notation), the next three columns show which primes
(dividing $N$) of the curve are supercuspidal, Steinberg and ramified
principal series, respectively. The next row gives the chosen $\omega$
(that determines the order in the imaginary quadratic field), and
which primes give rise to Cartan non-split groups (the remaining are
classical ones). It is easy to see that in each example the Cartan-Heegner condition is satisfied.  Then, we list the matrices
$M_{i}:=A_i N_{i} {A_{i}}^{-1}$ for some $\vec{\varepsilon}$. The next
column contains the first Fourier coefficient (where we use the
notation $\zeta_i :=\xi_N^i + \bar{\xi_N}^i$, and a vector
$[a_1, \ldots,a_N]$ means $a_1 \zeta_1 + \dots + a_N \zeta_N$), and
the last column gives the Manin constant $c$ for the optimal quotient.

\begin{table}[h]
\begin{tabular}{||r||c|c|c||c|r||c||c|c||}
  \hline
  EC & Sc & St & Ps& $\omega$ & $C_{ns}$  & $M_i$ & $b_1$& c \\
  \hline
  121b &\mbox{\footnotesize $\{11\}$} & $\emptyset$ & $\emptyset$ & \mbox{\small $\frac{1+\sqrt{-3}}{2}$} & \mbox{\footnotesize$\{11\}$} &$\left(\begin{smallmatrix} 6 & -31\\ 1 & -5 \end{smallmatrix}\right)$ & \mbox{\scriptsize $[-3,-1,-5,-4,2]$} & $\frac{1}{11}$\\[1ex]
  \hline

225a &\mbox{\footnotesize$\{3,5\}$}& $\emptyset$ & $\emptyset$ & \mbox{\small $\frac{1+\sqrt{-91}}{2}$} & \mbox{\footnotesize$\{3\}$}& $\left(\begin{smallmatrix}1 & -23\\ 1& 0\end{smallmatrix} \right)$ & \mbox{\scriptsize $1$} & 1 \\[1ex]
& & & &\mbox{\small $\frac{3+\sqrt{-91}}{10}$}&\mbox{\footnotesize $\{3\}$}&  $\left(\begin{smallmatrix}2 & -5\\ 5 & -1\end{smallmatrix} \right)$ & \mbox{\scriptsize $1$} & \\
\hline
225a & \mbox{\footnotesize$\{3,5\}$} & $\emptyset$ & $\emptyset$ & \mbox{\small $\frac{1+\sqrt{-7}}{2}$} & \mbox{\footnotesize$\{3,5\}$} & $\left(\begin{smallmatrix}8 & -58\\ 1 & -7\end{smallmatrix} \right)$ & \mbox{\footnotesize $\frac{1-\sqrt{5}}{2}$} & 1\\
\hline
289a & $\emptyset$ & \mbox{\footnotesize$\{17\}$} & $\emptyset$&\mbox{\small $\frac{1+\sqrt{-3}}{2}$} & \mbox{\footnotesize$\{17\}$} & $\left(\begin{smallmatrix}9 & -73\\ 1 & 8\end{smallmatrix} \right)$ & \mbox{\scriptsize $[-6,-7,-4,-1,-5,-2,-4,-5]$} & $\frac{1}{17}$\\[1ex]
\hline
1617a & $\emptyset$ & $\emptyset$ & \mbox{\footnotesize$\{7\}$} & \mbox{\footnotesize $\sqrt{-2}$} & \mbox{\footnotesize$\{7\}$}& $\left(\begin{smallmatrix} 14 & -6\\ 33 & -14\end{smallmatrix} \right)$ & \mbox{\scriptsize$[-2,-1,-4]$} & $\frac{1}{7}$\\[1ex]
\hline
49a & \mbox{\footnotesize$\{7\}$}   &$\emptyset$ &$\emptyset$ & \mbox{\footnotesize $\frac{1+\sqrt{-11}}{2}$} & \mbox{\footnotesize$\{7\}$}& $\left(\begin{smallmatrix} 4 & -15\\ 1 & -3\end{smallmatrix} \right)$ & \mbox{\footnotesize $\sqrt{-7}$} & $\frac{1}{7}$\\[1ex]
\hline
\end{tabular}
\caption{Examples of the $q$-expansion and related computational data}
\label{table:curvesdata}
\end{table}

\begin{rem}
  In all the examples of Table~\ref{table:curvesdata} but the last
  one, the optimal quotient coincides with the strong Weil curve. In
  the last example, the optimal quotient corresponds to the curve $49\text{a}2$
  in Cremona's notation. 
\end{rem}

In Table~\ref{table:points} we show the points constructed on the
curves of Table~\ref{table:curvesdata} and the multiple of the
generator obtained (up to torsion). Note that in the last case, the
curve has rank $0$ over $\QQ$, and this is why the point is not
rational.

\begin{table}[h]
\begin{tabular}{||r||c|c|c||}
\hline
EC &$K$& P & $m_P$\\
\hline
121b & $\QQ(\sqrt{-3})$ & $(\frac{2411156245}{(37062)^2},-\frac{52866724475375}{(37602)^3})$ & 15\\
\hline
225a& $\QQ(\sqrt{-91})$& $(1,1)$ & 1\\
\hline
225a & $\QQ(\sqrt{-7})$ & $(-1,0)$ & 2\\
\hline 
289a & $\QQ(\sqrt{-3})$ & $(-\frac{15858973521095}{1083383^2}, -\frac{22895413346586388187}{1083383^3})$ & 3\\[1ex]
\hline
1617a & $\QQ(\sqrt{-2})$& $(\frac{3702}{17^2},\frac{184078}{17^3})$ & 3\\
\hline
49a & $\QQ(\sqrt{-11})$& $(\frac{1261982}{11 (127)^2},- \frac{680991}{11 (127)^2} - \frac{327847275}{11^2 (127)^3} \sqrt{-11}  )$ & 3\\
\hline
\end{tabular}
\caption{Heegner points constructed}
\label{table:points}
\end{table}

\appendix
\section{The principal series case computation}
\label{app:1}

The purpose of this short Appendix is to show how the work \cite{Tim} (in particular example $5$)
allows to, given an elliptic curve $E$ with a ramified principal
series at $p$, compute the character to twist by, and the local $p$-th
Fourier coefficient of the forms $h_i$ in
Theorem~\ref{thm:combinationgeneral}. We thank Tim Dokchitser for
explaining us some details of the algorithm. 

\begin{enumerate}
\item Compute $v_p$ = the valuation at $p$ of the discriminant of
  $E$. The order of the character is
  $e:= \frac{12}{\gcd(12:v_p)}$. 
\item Let $L=\QQ(x)/(x^e-p)$. Then, $E$ attains  good reduction at
  the prime ideal $(x)$. Compute the characteristic polynomial
  $\chi_L(t)=t^2- a_pt + p$ of Frobenius at such prime ideal by
  counting the number of points over the finite field (this is
  implemented in SAGE or Magma). The two roots are the $p$-th
  coefficients we are looking for (since there are two forms, conjugate
  to each other), but we need to match each root with its
  corresponding character.
\item Let $g$ be a generator of $\FF_p^\times$, and let
  $L'=\QQ(x)/(x^e-g\cdot p)$. As before, compute the characteristic
  polynomial $\chi_{L'}(t)$ for the prime ideal $(x)$ (the curve is
  again unramified). Then the product of a root of $\chi_L(t)$ multiplied
  by the correct character (evaluated at $g$) must be a root of
  $\chi_{L'}(t)$.
\end{enumerate}

\bibliographystyle{alpha}
\bibliography{biblio}

\begin{thebibliography}{{PAR}14}

\bibitem[AWL78]{Atkin-Li}
AOL Atkin and Wein-Ch'ing Winnie~Li.
\newblock Twists of newforms and pseudo-eigenvalues of w-operators.
\newblock {\em Inventiones mathematicae}, 48(3):221--243, 1978.

\bibitem[Che98]{Chen}
Imin Chen.
\newblock The {J}acobians of non-split {C}artan modular curves.
\newblock {\em Proc. London Math. Soc. (3)}, 77(1):1--38, 1998.

\bibitem[Dar04]{Dar04}
Henri Darmon.
\newblock {\em Rational points on modular elliptic curves}, volume 101 of {\em
  CBMS Regional Conference Series in Mathematics}.
\newblock Published for the Conference Board of the Mathematical Sciences,
  Washington, DC; by the American Mathematical Society, Providence, RI, 2004.

\bibitem[DD11]{Tim}
Tim Dokchitser and Vladimir Dokchitser.
\newblock Euler factors determine local weil representations.
\newblock {\em arXiv:1112.4889}, 2011.

\bibitem[DS05]{MR2112196}
Fred Diamond and Jerry Shurman.
\newblock {\em A first course in modular forms}, volume 228 of {\em Graduate
  Texts in Mathematics}.
\newblock Springer-Verlag, New York, 2005.

\bibitem[dSE00]{SmEd}
Bart de~Smit and Bas Edixhoven.
\newblock Sur un r\'esultat d'{I}min {C}hen.
\newblock {\em Math. Res. Lett.}, 7(2-3):147--153, 2000.

\bibitem[Edi96]{Edi}
Bas Edixhoven.
\newblock On a result of imin chen.
\newblock {\em arXiv:alg-geom/9604008}, 1996.

\bibitem[Gro84]{Gr84}
Benedict~H. Gross.
\newblock Heegner points on {$X_0(N)$}.
\newblock In {\em Modular forms ({D}urham, 1983)}, Ellis Horwood Ser. Math.
  Appl.: Statist. Oper. Res., pages 87--105. Horwood, Chichester, 1984.

\bibitem[Gro91]{Gr89}
Benedict~H. Gross.
\newblock Kolyvagin's work on modular elliptic curves.
\newblock In {\em {$L$}-functions and arithmetic ({D}urham, 1989)}, volume 153
  of {\em London Math. Soc. Lecture Note Ser.}, pages 235--256. Cambridge Univ.
  Press, Cambridge, 1991.

\bibitem[GZ86]{GZ}
Benedict~H. Gross and Don~B. Zagier.
\newblock Heegner points and derivatives of {$L$}-series.
\newblock {\em Invent. Math.}, 84(2):225--320, 1986.

\bibitem[Lan87]{Lang}
Serge Lang.
\newblock {\em Elliptic functions}, volume 112 of {\em Graduate Texts in
  Mathematics}.
\newblock Springer- Verlag, New York, NY, 1987.

\bibitem[Maz78]{Maz78}
B.~Mazur.
\newblock Rational isogenies of prime degree (with an appendix by {D}.
  {G}oldfeld).
\newblock {\em Invent. Math.}, 44(2):129--162, 1978.

\bibitem[Pac13]{Pacetti}
Ariel Pacetti.
\newblock On the change of root numbers under twisting and applications.
\newblock {\em Proc. Amer. Math. Soc.}, 141(8):2615--2628, 2013.

\bibitem[{PAR}14]{PARI}
{PARI~Group}, Bordeaux.
\newblock {\em {PARI/GP version {\tt 2.7.0}}}, 2014.
\newblock available from \url{http://pari.math.u-bordeaux.fr/}.

\bibitem[Raj98]{rajan1998strong}
CS~Rajan.
\newblock On strong multiplicity one for l-adic representations.
\newblock {\em International Mathematics Research Notices}, 1998(3):161--172,
  1998.

\bibitem[RW14]{Rebolledo}
Marusia Rebolledo and Christian Wuthrich.
\newblock A moduli interpretation for the non-split cartan modular curve.
\newblock {\em arXiv:1402.3498}, 2014.

\bibitem[Ser67]{Serre2}
J.-P. Serre.
\newblock Complex multiplication.
\newblock In {\em Algebraic {N}umber {T}heory ({P}roc. {I}nstructional {C}onf.,
  {B}righton, 1965)}, pages 292--296. Thompson, Washington, D.C., 1967.

\bibitem[Ser97]{Serre}
Jean-Pierre Serre.
\newblock {\em Lectures on the {M}ordell-{W}eil theorem}.
\newblock Aspects of Mathematics. Friedr. Vieweg \& Sohn, Braunschweig, third
  edition, 1997.
\newblock Translated from the French and edited by Martin Brown from notes by
  Michel Waldschmidt, With a foreword by Brown and Serre.

\bibitem[Shi94]{Shimura}
Goro Shimura.
\newblock {\em Introduction to the arithmetic theory of automorphic functions},
  volume~11 of {\em Publications of the Mathematical Society of Japan}.
\newblock Princeton University Press, Princeton, NJ, 1994.
\newblock Reprint of the 1971 original, Kan{\^o} Memorial Lectures, 1.

\bibitem[Zha01]{ZH}
Shou-Wu Zhang.
\newblock Gross-{Z}agier formula for {${\rm GL}_2$}.
\newblock {\em Asian J. Math.}, 5(2):183--290, 2001.

\bibitem[Zha04]{Zh2}
Shou-Wu Zhang.
\newblock Gross-{Z}agier formula for {$\rm GL(2)$}. {II}.
\newblock In {\em Heegner points and {R}ankin {$L$}-series}, volume~49 of {\em
  Math. Sci. Res. Inst. Publ.}, pages 191--214. Cambridge Univ. Press,
  Cambridge, 2004.

\end{thebibliography}
\end{document}